\documentclass[12pt,a4paper]{amsart}

\usepackage{lscape}
\usepackage{amscd}
\usepackage{amssymb}
\usepackage{amsmath}

\def\so{\mathfrak{so}}
\def\SO{\mathrm{SO}}

\def\su{\mathfrak{su}}

\def\hh{\mathfrak{h}}
\def\tt{\mathfrak{t}}
\def\mm{\mathfrak{m}}
\def\gg{\mathfrak{g}}

\def\Cl{\mathrm{Cl}}
\def\spin{\mathfrak{spin}}
\def\Spin{\mathrm{Spin}}

\def\G{\mathrm{G}}

\def\rr{r}
\def\f{\varphi}
\def \PM{\mathbb{P}}
\def \RM{\mathbb{R}}
\def \Ca{\mathbb{O}}

\def \ZM{\mathbb{Z}}
\def \CM{\mathbb{C}}

\def \HM{\mathbb{H}}

\def\E{\mathrm{E}}
\def\F{\mathrm{F}}
\def\A{\mathcal{A}}
\def\G{\mathcal{G}}
\def\B{\mathcal{B}}
\def\1{\mathbf{1}}
\def\#{\sharp}
\def\ad{\mathrm{ad}}

\def\C{\mathbb{C}}

\def\e{\varepsilon}
\def\End{\mathrm{End}\,}

\def\h{\mathfrak{h}}
\def\H{\mathbb{H}}

\def\K{\mathbb{K}}
\def\id{\mathrm{id}}

\def\l{\lambda}
\def\m{\mathfrak{m}}

\def\R{\mathbb{R}}
\def\rk{\mathrm{rk}}

\def\S{\mathrm{\Sigma}}

\def\W{\mathfrak{W}}

\def\<#1,#2>{\langle\,#1,\,#2\,\rangle}

\def\beq{\begin{equation}}
\def\eeq{\end{equation}}

\def\norm(#1){\|#1\|^2}
\def\rectangle(#1,#2)[#3,#4]#5{
 \multiput(#1,#2)(#3,0)2{\line(0,1){#4}}\multiput(#1,#2)(0,#4)2{\line(1,0){#3}}
 \put(#1,#2){\vbox to #4pt{\hbox to #3pt{\hfill}\vfill}}}
\def\recttext(#1,#2)[#3,#4]#5{\put(#1,#2)
 {\vbox to #4pt{\vfill\hbox to #3pt{\hss#5\hss}\vfill}}}
\newtheorem{Lemma}{Lemma}[section]
\newtheorem{Proposition}[Lemma]{Proposition}

\newtheorem{Theorem}[Lemma]{Theorem}

\theoremstyle{definition}

\newtheorem{Definition}[Lemma]{Definition}
\newtheorem{Remark}[Lemma]{Remark}

\title{Higher rank homogeneous Clifford structures}

\author{Andrei Moroianu and Mihaela Pilca}
\thanks{This work was supported 
by the contract ANR-10-BLAN 0105 ``Aspects Conformes de la
G{\'e}om{\'e}trie''. The second-named author acknowledges 
also partial support from the 
CNCSIS grant PNII IDEI contract 529/2009 and thanks
the Centre de Math{\'e}matiques de 
l'{\'E}cole Polytechnique for hospitality 
during the preparation of this work.}

\address{Andrei Moroianu \\ Universit\'e de Versailles-St Quentin \\
Laboratoire de Math\'ematiques \\ UMR 8100 du CNRS\\
45 avenue des \'Etats-Unis\\
78035 Versailles, France }
\email{am@math.polytechnique.fr}

\address{Mihaela Pilca\\Fakult\"at f\"ur Mathematik\\
Universit\"at Regensburg\\Universit\"atsstr. 31 
D-93040 Regensburg, Germany
\emph{and} 
Institute of Mathematics ``Simion Stoilow" of the Romanian Academy, 
21, Calea Grivitei Str.
010702-Bucharest, Romania}
\email{Mihaela.Pilca@mathematik.uni-regensburg.de}

\begin{document}

\begin{abstract}
We give an upper bound for the rank $r$ of homogeneous (even) Clifford structures on compact 
manifolds of non-vanishing Euler characteristic. More precisely, we show that if $r=2^a\cdot b$ with $b$ odd, 
then $r\le 9$ for $a=0$, $r\le 10$ for $a=1$, $r\le 12$ for $a=2$ and $r\le 16$ for $a\ge 3$. 
Moreover, we describe the four limiting cases and show that there is exactly one solution in each case.

\medskip


\smallskip

\end{abstract}
\maketitle

\section{Introduction}

The notion of (even) Clifford structures on Riemannian manifolds
was introduced in \cite{ms}, motivated by the study of Riemannian manifolds
with non-trivial curvature constancy (cf. \cite{gray}). They generalize almost Hermitian and
quaternionic-Hermitian structures and are in some sense dual to spin structures. More precisely:

\begin{Definition}[\cite{ms}] \label{d1}
A rank $\rr$ {\em Clifford structure} ($\rr\ge 1$) on a Riemannian
manifold $(M^n,g)$ is an oriented rank $\rr$ Euclidean bundle
$(E,h)$ over $M$ together with an algebra bundle
morphism $\f:\Cl(E,h)\to \End(TM)$
which maps $E$ into the bundle of skew-symmetric endomorphisms
$\End^-(TM)$.

A rank $\rr$ {\em even Clifford structure} ($\rr\ge 2$) on
$(M^n,g)$ is an oriented rank $\rr$ Euclidean bundle
$(E,h)$ over $M$ together with an algebra bundle
morphism $\f:\Cl^0(E,h)\to \End(TM)$
which maps $\Lambda^2E$ into the bundle of skew-symmetric endomorphisms
$\End^-(TM)$.
\end{Definition}

It is easy to see that every rank $\rr$ Clifford structure is in particular
a rank $\rr$ even Clifford structure, so the latter notion is more flexible.

In general, there exists no upper bound for the rank of a Clifford structure. In fact, Joyce
provided a method (cf. \cite{jo}) to construct non-compact manifolds with arbitrarily large Clifford
structures. However, in many cases the rank is bounded by above. For instance, the Riemannian
manifolds carrying {\em parallel} (even) Clifford structures (in the sense that $(E, h)$ has a metric
connection making the Clifford morphism $\f$ parallel) were classified in \cite{ms} and it turns out
that the rank of a parallel Clifford structure is bounded by above if the manifold is nonflat:
Every parallel Clifford structure has rank $\rr\le 7$ and every parallel even Clifford structure
has rank $\rr\le 16$ (cf. \cite[Thm. 2.14 and 2.15]{ms}).
The list of manifolds with parallel even Clifford structure of rank $\rr\ge 9$ only contains
four entries, the so-called Rosenfeld's elliptic projective planes $\Ca \PM^2$,
$(\CM\otimes
\Ca) \PM^2$, $(\HM\otimes \Ca) \PM^2$ and $(\Ca\otimes \Ca) \PM^2$, which
are inner symmetric spaces associated to
the exceptional simple Lie groups $\F_4$, $\E_6$, $\E_7$ and $\E_8$ (cf. \cite{r}) and have Clifford rank
$\rr=9,\ 10,\ 12$ and $16$, respectively.

A natural related question is then to look for
{\em homogeneous} (instead of parallel) even Clifford structures on homogeneous
spaces $M=G/H$. We need to make some restrictions on $M$ in order to obtain relevant results. 
On the first hand, we assume $M$ to be compact (and thus $G$ and $H$ are compact, too). 
On the other hand, we need to assume that $H$ is not too small. For example, in the degenerate case when $H$ is 
just the identity of $G$, the tangent bundle of $M=G$ is trivial, and the unique obstruction 
for the existence of a rank $r$ (even) Clifford structure is that the dimension of $G$ has
to be a multiple of the dimension of the irreducible representation of the Clifford algebra $\Cl_r$ or $\Cl_r^0$.
At the other extreme, we might look for homogeneous spaces $M=G/H$ with $\rk(H)=\rk(G)$, or, equivalently, $\chi(M)\ne 0$.
The main advantage of this assumption is that we can choose a common 
maximal torus of $H$ and $G$ and identify the root system of $H$ with a subset of the root system of $G$.

In this setting, the system of roots of $G$ is made up of the
system of roots of $H$ and the weights of the (complexified) isotropy representation,
which are themselves related to the weights of some spinorial representation if $G/H$ carries a 
homogeneous even Clifford structure. We then show that the very special configuration of the
weights of the spinorial representation $\Sigma_r$ is not compatible with the usual integrality 
conditions of root systems, provided that $r$ is large enough.

The main results of this paper are Theorem \ref{upperbd}, where we obtain upper bounds on $r$ 
depending on its 2-valuation, and Theorem \ref{h-types}, where we study the limiting cases 
$\rr=9,\ 10,\ 12$ and $16$ and show that they correspond to the symmetric spaces 
$\mathrm{F}_4/\mathrm{Spin}(9)$, $\mathrm{E}_6/(\mathrm{Spin}(10)\times\mathrm{U}(1)/\ZM_4)$, 
$\mathrm{E}_7/\mathrm{Spin}(12)\cdot\mathrm{SU}(2)$ and  $\mathrm{E}_8/(\mathrm{Spin}(16)/\mathbb{Z}_2)$.

We believe that our methods could lead to a complete classification of homogeneous 
Clifford structures of rank $r\ge 3$ on compact manifolds with non-vanishing Euler 
characteristic, eventually showing that they are all symmetric, thus parallel 
(cf. \cite[Table 2]{ms}), but a significantly larger amount of work is needed, especially for lower ranks.

\section{Preliminaries on Lie algebras and root systems}

For the basic theory of root systems we refer to \cite{a69} and \cite{s}.

\begin{Definition}\label{sysroot}
A set $\mathcal{R}$ of vectors in a Euclidean space $(V,\<\cdot,\cdot>)$ is called a
\emph{system of roots}  if it satisfies the following
conditions:
\begin{description}
\item[R1] $\mathcal{R}$ is finite, $\mathrm{span}(\mathcal{R})=V$,
  $0\notin \mathcal{R}$.
\item[R2] If $\alpha\in \mathcal{R}$, then the only multiplies of
  $\alpha$ in $\mathcal{R}$ are $\pm\alpha$.
\item[R3] $\frac{2\<\alpha,\beta>}{\<\alpha,\alpha>}\in\mathbb{Z}$,
  for all $\alpha, \beta\in \mathcal{R}$.
\item[R4] $s_\alpha:\mathcal{R}\to \mathcal{R}$, for all $\alpha\in
  \mathcal{R}$ ($s_\alpha$ is the reflection $s_\alpha:V\to
  V$, $s_\alpha(v):=v-\frac{2\<\alpha,v>}{\<\alpha,\alpha>}\alpha$).
\end{description}
\end{Definition}

\begin{Remark}[Properties of root systems]
Let $\mathcal{R}$ be a system of roots. If
$\alpha,\beta\in\mathcal{R}$ such that $\beta\neq\pm\alpha$ and
$\norm(\beta)\geq\norm(\alpha)$, then
\begin{equation}\label{prodscal}
\frac{2\<\alpha,\beta>}{\<\beta,\beta>}\in\{0,\pm1\}.
\end{equation}
If $\<\alpha,\beta>\neq 0$, then the following values are possible:
\begin{equation}\label{normscal}
\left(\frac{\norm(\beta)}{\norm(\alpha)},
  \frac{2\<\alpha,\beta>}{\<\alpha,\alpha>}\right)\in\{(1,\pm
1),(2,\pm 2),(3,\pm 3)\}.
\end{equation}
Moreover, in this case, it follows that
\begin{equation}\label{newroots}
\beta-\mathrm{sgn}\left(\frac{2\<\alpha,\beta>}
  {\<\alpha,\alpha>}\right)k\alpha\in\mathcal{R},  \quad \text{ for }
k\in\mathbb{Z}, 1\leq k\leq \biggl|\frac{2\<\alpha,\beta>}{\<\alpha,\alpha>}\biggr|.
\end{equation}
\end{Remark}

We shall be interested in special subsets of systems of roots and
consider the following notions.

\begin{Definition}\label{subsys}
A set $\mathcal{P}$ of vectors in a Euclidean space $(V,\<\cdot,\cdot>)$
is called a \emph{subsystem of roots} if it generates $V$ and is contained in a system of roots  of $(V,\langle\cdot,\cdot\rangle)$.
\end{Definition}

It is clear that any subsystem
of roots $\mathcal{P}$ is included into a minimal system of roots
(obtained by taking all possible reflections), which we denote by
$\overline{\mathcal{P}}$.

Let $\mathcal{P}$ be a subsystem of roots of $(V,\<\cdot,\cdot>)$. 
An {\em irreducible component} of $\mathcal{P}$ is a minimal non-empty subset 
$\mathcal{P}'\subset \mathcal{P}$ such that $\mathcal{P}'\perp (\mathcal{P}\setminus \mathcal{P}')$.
By rescaling the scalar product $\<\cdot,\cdot>$ on the subspaces generated by
the irreducible components of $V$ one can always assume that the root of maximal length of each irreducible component
of ${\mathcal{P}}$ has norm equal to $1$.

\begin{Definition}\label{asubsys}
A subsystem of roots $\mathcal{P}$ in a Euclidean space $(V,\<\cdot,\cdot>)$ is called
\emph{admissible} if $\overline{\mathcal{P}}\setminus\mathcal{P}$ is a
system of roots.
\end{Definition}

For any $q\in\mathbb{Z}$, $q\geq 1$, let $\mathcal{E}_q$ denote the
set of all $q$-tuples $\varepsilon:=(\varepsilon_1, \dots,
\varepsilon_q)$ with $ \varepsilon_j\in\{\pm1\},\ 1\le j\le q$.
The following result will be used several times in the next section.

\begin{Lemma}\label{admis}
Let $q\in\mathbb{Z}$, $q\geq 1$ and $\{\beta_j\}_{j=\overline{1,q}}$ 
be a set of linearly independent
vectors in a Euclidean space $(V,\<\cdot,\cdot>)$. If
$\mathcal{P}\subset\{\overset{q}{\underset{j=1}
{\sum}}\varepsilon_j\beta_j\}_{\varepsilon\in\mathcal{E}_q}$
is an admissible subsystem of roots, then any two vectors in
$\mathcal{P}$ of different norms must be orthogonal.
\end{Lemma}

\begin{proof}
Assuming the existence of two non-orthogonal vectors of different
norms,  we construct two roots in $\overline{\mathcal{P}}\setminus\mathcal{P}$ 
whose difference is a root in $\mathcal{P}$, thus contradicting the
assumption on $\mathcal{P}$ to be admissible. More precisely, suppose
that $\alpha,\alpha'\in\mathcal{P}$, $\alpha'\neq\alpha$, such that $\<\alpha,\alpha'>> 0$
(a similar argument works if $\<\alpha,\alpha'><0$) and
$\norm(\alpha')>\norm(\alpha)$. From \eqref{normscal}, it follows that
either $\norm(\alpha')=2\norm(\alpha)$ and
$\<\alpha,\alpha'>=\norm(\alpha)$ or $\norm(\alpha')=3\norm(\alpha)$
and $\<\alpha,\alpha'>=\frac{3}{2}\norm(\alpha)$. In both cases
$\bigl|\frac{2\<\alpha',\alpha>}{\<\alpha,\alpha>}\bigr|\geq 2$ and 
\eqref{newroots} implies that $\alpha'-\alpha, \alpha'-2\alpha\in
\overline{\mathcal{P}}$.

We first check that $\alpha'-\alpha, \alpha'-2\alpha\notin
\mathcal{P}$. The coefficients of $\beta_j$ in $\alpha'-\alpha$ and in
$\alpha'-2\alpha$ may take the values $\{0,\pm2\}$, respectively
$\{\pm1, \pm3\}$, for all $j=1,\dots,q$. Since
$\{\beta_j\}_{j=\overline{1,q}}$ are linearly independent, it follows
that $\alpha'-\alpha\notin
\{\overset{q}{\underset{j=1}{\sum}}\varepsilon_j\beta_j\}_
{\varepsilon\in\mathcal{E}_q}$.
Moreover, 
$\alpha'-2\alpha\in\{\overset{q}{\underset{j=1}{\sum}}
\varepsilon_j\beta_j\}_{\varepsilon\in\mathcal{E}_q}$
if and only if the coefficients of each $\beta_j$ in $\alpha'$
and $\alpha$ are equal,
\emph{i.e.}  $\alpha'=\alpha$, which is not possible.

On the other hand,
$\<\alpha'-\alpha,\alpha'-2\alpha>\in\{\norm(\alpha),
\frac{1}{2}\norm(\alpha)\}$ and from \eqref{newroots} and the
admissibility of $\mathcal{P}$, it follows that
$\alpha\in\overline{\mathcal{P}}\setminus\mathcal{P}$, yielding a
contradiction and finishing the proof.  
\end{proof}

Let $G$ be a compact semi-simple Lie group with Lie algebra $\mathfrak{g}$ 
endowed with an $\ad_\gg$-invariant scalar product. 
Fix a Cartan subalgebra $\tt\subset\gg$ and let $\mathcal{R}(\gg)\subset \mathfrak{t}^*$ denote
its system of roots. It is well known that $\mathcal{R}(\gg)$ satisfies the conditions 
in Definition \ref{sysroot}. Conversely,
every set of vectors satisfying the conditions in Definition \ref{sysroot} is the root system of 
a unique semi-simple Lie algebra of compact type.

If $H$ is a closed subgroup of $G$ with $\rk(H)=\rk(G)$, then one may assume that its Lie algebra $\hh$ contains $\tt$, 
so the system of roots $\mathcal{R}(\mathfrak{g})$
is the disjoint union of the root system $\mathcal{R}(\mathfrak{h})$ and the set $\mathcal{W}$ 
of weights of the complexified isotropy representation of the homogeneous space $G/H$. This follows from the fact that the isotropy
representation is given by the restriction to $H$ of the adjoint
representation of $\mathfrak{g}$. 
\begin{Lemma}\label{adm}
The set $\mathcal{W}\subset \mathfrak{t}^*$ is an admissible subsystem of roots. 
\end{Lemma}
\begin{proof}
 Indeed,
$\overline{\mathcal{W}}\setminus\mathcal{W}=\overline{\mathcal{W}}\cap\mathcal{R}(\mathfrak{h})$,
whence $\overline{\overline{\mathcal{W}}\setminus\mathcal{W}}\subset \overline{\mathcal{W}}\cap\overline{\mathcal{R}(\mathfrak{h})}=
\overline{\mathcal{W}}\cap\mathcal{R}(\mathfrak{h})=\overline{\mathcal{W}}\setminus\mathcal{W}$.
\end{proof}

We will now prove a few general results about Lie algebras which will be needed later
on.

\begin{Lemma}\label{sumroots}
Let $\mathfrak{h}_1$ be a Lie subalgebra of a Lie algebra
$\mathfrak{h}_2$ of compact type having the same rank. If $\alpha,\beta\in
\mathcal{R}(\mathfrak{h}_1)$ such that
$\alpha+\beta\in\mathcal{R}(\mathfrak{h}_2)$, then
$\alpha+\beta\in\mathcal{R}(\mathfrak{h}_1)$.
\end{Lemma}

\begin{proof}
We first recall a general result about roots.
Let $\hh$ be a Lie algebra of compact type and $\mathfrak{t}$
a fixed Cartan subalgebra in $\mathfrak{h}$.
For any $\alpha\in\mathfrak{t}^*$, let $(\mathfrak{h})_{\alpha}$ denote 
the intersection of the nilspaces
of the operators $\mathrm{ad}(A)-\alpha(A)$ acting on $\mathfrak{h}$,
with $A$ running over $\mathfrak{t}$.
By definition, $\alpha$ is a root of $\mathfrak{h}$ if and only if  $(\mathfrak{h})_{\alpha}\neq\{0\}$.
Moreover, the Jacobi identity shows that $[(\mathfrak{h})_{\alpha}, 
(\mathfrak{h})_{\beta}]\subseteq(\mathfrak{h})_{\alpha+\beta}$.
It is well known that in this case the space $(\mathfrak{h})_{\alpha}$ is $1$-dimensional.
Moreover, by \cite[Theorem~A, p.~48]{s}, there exist generators
$X_{\alpha}$ of $(\mathfrak{h})_{\alpha}$
such that for any $\alpha,\beta\in\mathcal{R}(\mathfrak{h})$
with $\alpha+\beta\in\mathcal{R}(\mathfrak{h})$,
the following relation holds:
$[X_\alpha,X_\beta]=\pm(q+1)X_{\alpha+\beta}$, where $q$ is the largest integer $k$
such that $\beta-k\alpha$ is a root.
In particular, if $\alpha+\beta\in\mathcal{R}(\mathfrak{h})$,
then $[(\mathfrak{h})_{\alpha}, (\mathfrak{h})_{\beta}]=(\mathfrak{h})_{\alpha+\beta}$.

Let now $\mathfrak{t}$ be a fixed Cartan subalgebra in both $\mathfrak{h}_1$
and $\mathfrak{h}_2$ (this is possible because $\rk(\mathfrak{h}_1)=\rk(\mathfrak{h}_2)$)
and let $\alpha,\beta\in\mathcal{R}(\mathfrak{h}_1)\subseteq\mathcal{R}(\hh_2)$,
such that $\alpha+\beta\in\mathcal{R}(\mathfrak{h}_2)$.
The above result applied to $\hh_2$ implies:
$\{0\}\neq(\hh_2)_{\alpha+\beta}=[(\hh_2)_\alpha, (\hh_2)_\beta]=
[(\hh_1)_\alpha, (\hh_1)_\beta]\subseteq (\hh_1)_{\alpha+\beta}$,
where we use that $(\hh_1)_\alpha=(\hh_2)_\alpha$
for any $\alpha\in \mathcal{R}(\mathfrak{h}_1)\subseteq\mathcal{R}(\hh_2)$.
Thus, $(\hh_1)_{\alpha+\beta}\neq\{0\}$,
\emph{i.e.} $\alpha+\beta\in\mathcal{R}(\mathfrak{h}_1)$.
\end{proof}

We will also need the following result, whose proof is straightforward.

\begin{Lemma}\label{mixedroots}
(i) Let $k\geq 2$ and let $\mathfrak{h}$ be a Lie algebra of compact type written as an orthogonal direct sum of $k$ Lie algebras:
$\mathfrak{h}=\overset{k}{\underset{i=1}{\oplus}}\mathfrak{h}_i$ with respect to some $\ad_\hh$-invariant scalar product $\<\cdot,\cdot>$ on $\hh$.
Then, identifying each Lie algebra $\hh_i$ with its dual using $\<\cdot,\cdot>$ we have
$\mathcal{R}(\mathfrak{h})=\overset{k}{\underset{i=1}{\bigcup}}\mathcal{R}(\mathfrak{h}_i)$.
In particular, every root of $\hh$ lies in one component $\hh_i$. 

(ii) Let $\alpha$ and $\beta$ be two roots of $\hh$. If there exists a sequence of roots 
$\alpha_0:=\alpha,\alpha_1,\ldots,\alpha_n:=\beta$ ($n\ge 1$) such that $\<\alpha_i,\alpha_{i+1}>\ne 0$ for $0\le i\le n-1$, then
$\alpha$ and $\beta$ belong to the same component $\hh_i$.
\end{Lemma}

\section{The isotropy representation of homogeneous manifolds with Clifford structure}

Let $M=G/H$ be a compact homogeneous space. Denote by $\hh$ and $\gg$ the Lie 
algebras of $H$ and $G$ and by $\m$ the orthogonal complement of $\hh$ in $\gg$ 
with respect to some $\ad_\gg$-invariant scalar product on $\gg$. The restriction 
to $\mm$ of this scalar product defines a homogeneous Riemannian metric $g$ on $M$.
Since from now on we will exclusively consider even Clifford structures, and in 
order to simplify the terminology, we will no longer use the word ``even" and 
make the following:

\begin{Definition}
A {\em homogeneous Clifford structure}  of rank $\rr\ge 2$ on a Riemannian homogeneous
space $(G/H,g)$ is an orthogonal representation $\rho:H\to\SO(r)$ and an
$H$-equivariant representation $\varphi: \so(r)\to
\mathrm{End}^{-}(\mathfrak{m})$ extending to an algebra representation
of the even real Clifford algebra $\Cl^0_r$ on $\mathfrak{m}$.
\end{Definition}

Any homogeneous Clifford structure defines in a tautological way an 
{\em even Clifford structure} on $(M,g)$ in the sense of Definition 
\ref{d1}, by taking $E$ to be the vector bundle associated to the 
$H$-principal bundle $G$ over $M$ via the representation $\rho$:
$$E=G\times_\rho \R^r.$$ 
In order to describe the isotropy representation of a homogeneous 
Clifford structure we need to recall some facts about Clifford algebras, 
for which we refer to \cite{lm}.

The even real Clifford algebra $\Cl^0_r$ is isomorphic to a matrix
algebra $\K(n_r)$ for $r\not\equiv 0 \mod 4$ and to a direct sum
$\K(n_r)\oplus\K(n_r)$ when $r$ is multiple of $4$. The field $\K\ (=\R,\
\C\ \hbox{or}\ \H)$ and the dimension $n_r$ depend on $r$ according to
a certain 8-periodicity rule. More precisely, $\K=\R$ for $r\equiv 0,\ 1,\ 7
\mod\ 8$,  $\K=\C$ for $r\equiv 2,\ 6 \mod\ 8$ and  $\K=\H$ for
$r\equiv 3,\ 4,\ 5
\mod\ 8$, and if we write $r=8k+q$, $1\le q\le 8$, then $n_r=2^{4k}$ for $1\le q
\le 4$, $n_r=2^{4k+1}$ for $q=5$, $n_r=2^{4k+2}$ for $q=6$ and
$n_r=2^{4k+3}$ for $q=7$ or $q=8$.

Let $\S_r$ and $\S_r^\pm$ denote the irreducible
representations of $\Cl_r^0$ for $r\not\equiv 0\mod 4$ and $r\equiv
0\mod
4$ respectively. From the above, it is clear that $\S_r$ (or $\S_r^\pm$) have
dimension $n_r$ over $\K$.

\begin{Lemma}\label{isorep} 
Assume that $M=G/H$ carries a rank $r$ homogeneous Clifford structure
and let $\iota:H\to \mathrm{Aut}(\mathfrak{m})$ denote the isotropy
representation of $H$.

(i) If $r$ is not a multiple of $4$, we denote by $\xi$ the spin
representation of $\so(r)=\spin(r)$ on the spin module $\Sigma_r$ and by
$\mu=\xi\circ\rho_*$ its composition with $\rho_*$. Then the
infinitesimal isotropy representation $\iota_*$ on
$\mathfrak{m}$ is isomorphic to  
$\mu\otimes_\K \l$ for some representation $\l$ of $\h$ over $\K$.

(ii) If $r$ is multiple of $4$, we denote by $\xi^\pm$ the half-spin
representations of $\so(r)=\spin(r)$ on the half-spin modules
$\Sigma_r^\pm$ and by $\mu_\pm=\xi^\pm\circ\rho_*$ their compositions
with $\rho_*$. Then the infinitesimal isotropy
representation $\iota_*$ on $\mathfrak{m}$ is isomorphic to
$\mu_+\otimes_\K \l_+\oplus\mu_-\otimes_\K \l_-$ for
some representations $\l_\pm$ of $\h$ over $\K$.
\end{Lemma}

\begin{proof}
(i) Consider first the case when $r$ is not a multiple of $4$.
By definition, the $H$-equivariant representation $\varphi: \so(r)\to
\mathrm{End}^{-}(\mathfrak{m})$ extends to an algebra representation
of the even Clifford algebra $\Cl^0_r\simeq \K(n_r)$ on
$\mathfrak{m}$. Since every algebra representation of the matrix
algebra $\K(n)$ decomposes in a direct sum of irreducible
representations, each of them isomorphic to the standard
representation on $\K^n$, we deduce that $\varphi$ is a direct sum of
several copies of $\Sigma_r$. In other words, $\mathfrak{m}$ is
isomorphic to $\Sigma_r\otimes_\K \K^p$ for some $p$, and $\varphi$ is
given by $\varphi(A)(\psi\otimes v)=(\xi(A)\psi)\otimes v.$ We now
study the isotropy representation $\iota_*$ on
$\mathfrak{m}=\Sigma_r\otimes_\K \K^p$. Note that when $\K=\H$ is
non-Abelian, some care is required in order to define the tensor
product of representations over $\K$.

The $H$-equivariance of $\varphi$ is equivalent to:
\[\iota(h)\circ\varphi(A)\circ\iota(h)^{-1}=\varphi(\rho(h)A),\qquad
\forall A\in\so(r), \forall h\in H.\]
Differentiating this relation at $h=1$ yields
\[\iota_*(X)\circ\varphi(A)-\varphi(A)\circ\iota_*(X)=
\varphi(\rho_*(X)A),\qquad
\forall A\in\so(r), \forall X\in \h.\]
On the other hand, 
\[\varphi(\rho_*(X)A)=\varphi([\rho_*(X),A])=
[\varphi(\rho_*(X)),\varphi(A)]=[\mu(X),\varphi(A)],\]
so
\beq\label{com}[\iota_*(X)-\mu(X),\varphi(A)]=0,\qquad \forall
A\in\so(r), \forall
X\in \h.\eeq
We denote by $\lambda:=\iota_*-\mu$. If $\{v_i\}$ denotes the standard
basis of $\K^p$ we introduce the maps
$\lambda_{ij}:\h\to\End_\K(\Sigma_r)$ by
\[\lambda(X)(\psi\otimes v_i)=\sum_{j=1}^p\lambda_{ji}(X)
(\psi)\otimes v_j.  \]
The previous relation shows that $\lambda_{ij}(X)$ commutes with the  
Clifford action $\xi(A)$ on $\Sigma_r$ for every $A\in\so(r)$, so it
belongs to $\K$. The matrix with entries $\lambda_{ij}(X)$
thus defines a Lie algebra representation $\lambda:\h\to \End_\K(\K^p)$
such that 
\[ \iota_*(X)(\psi\otimes v)=\mu(X)(\psi)\otimes v + \psi\otimes
\lambda(X)(v),\qquad \forall X\in h, \forall \psi\in\Sigma_r, \forall
v\in \K^p.\]
This proves the lemma in this case. 

(ii) If $r$ is multiple of $4$,
the even Clifford algebra $\Cl_r^0$ has two inequivalent algebra
representations $\Sigma_r^\pm$. One can write like before
$\mathfrak{m}=\Sigma^+\otimes_\K \K^p\oplus\Sigma^-\otimes_\K \K^{p'}$
for some $p,p'\ge 0$, and $\varphi$ is given by
$\varphi(A)(\psi^+\otimes v+\psi^-\otimes v')=(\xi^+(A)\psi^+)\otimes
v+(\xi^-(A)\psi^-)\otimes v'.$ The rest of the proof is similar, using
the fact that every endomorphism from $\Sigma_r^\pm$ to $\Sigma_r^\mp$
commuting with the Clifford action of $\so(r)$ vanishes. 
\end{proof}

Let us introduce the ideals $\hh_1:=\ker(\rho_*)$ and $\hh_2:=\ker(\lambda)$ of $\hh$.
Since the isotropy representation is faithful, $\hh_1\cap\hh_2=0$
and it is easy to see that $\hh_1$ is orthogonal to $\hh_2$ with respect to
the restriction to $\hh$ of any $\ad_\gg$-invariant scalar product. Denoting by $\h_0$ the
orthogonal complement of $\hh_1\oplus\hh_2$ in $\hh$ we obtain the following
orthogonal decomposition:
\begin{equation}\label{hspl}
\mathfrak{h}=\mathfrak{h}_0\oplus\mathfrak{h}_1\oplus\mathfrak{h}_2
\end{equation}
and the corresponding splitting of the Cartan subalgebra of
$\mathfrak{h}$:
$\mathfrak{t}=\mathfrak{t}_0\oplus\mathfrak{t}_1\oplus\mathfrak{t}_2$.

Lemma~\ref{isorep} yields further a description of the weights of the isotropy representation
of homogeneous spaces with Clifford structure. We assume from now on 
that $\rk(G)=\rk(H)$ and choose a common Cartan subalgebra $\tt\subset\hh\subset\gg$. The system of roots of
$\mathfrak{g}$ is then the disjoint union of the system of roots of $\mathfrak{h}$ and the
weights of the complexified isotropy representation. Since each weight is simple 
(cf. \cite[p. 38]{s}) we deduce that all weights of $\mm\otimes_\RM\C$ are simple.

If $r$ is not multiple of 4, Lemma~\ref{isorep} (i) shows that the isotropy 
representation $\m$ is isomorphic to $\mu\otimes_\K\lambda$ for some representations $\mu$ and $\lambda$ of
$\h$ over $\K$. In order to express $\mm\otimes_\RM\C$ it will be convenient to use the following convention: If $\nu$
is a representation over $\K$, we denote by $\nu^\CM$ the representation over $\CM$ given by
$$\begin{cases}
\nu^\CM=\nu\otimes_\RM\CM,\qquad& \hbox{if}\ \K=\R\\
\nu^\CM=\nu,& \hbox{if}\ \K=\C\\
\nu^\CM=\nu,& \hbox{if}\ \K=\H\\
\end{cases}$$
where in the last row $\nu$ is viewed as complex representation by fixing one of the complex structures. Using the fact that
if $\mu$ and $\lambda$ are quaternionic representations,
then there is a  natural isomorphism between $(\mu\otimes_\H\lambda)^\C$ and
$\mu^\C\otimes_\C\lambda^\C$, one can then write
\beq\label{mm1}\begin{cases}\mm\otimes_\R \CM=\mu^\C\otimes_\C\lambda^\C,\qquad& \hbox{if}\ \K=\R\ \hbox{or}\ \H\\
\mm\otimes_\R \CM=\mu^\C\otimes_\C\lambda^\C\oplus \bar\mu^\C\otimes_\C\bar\lambda^\C,& \hbox{if}\ \K=\C
\end{cases}\eeq

If $r$ is multiple of 4, then $\mm=\mu_+\otimes_\K\lambda_+\oplus \mu_-\otimes_\K\lambda_-$ by 
Lemma~\ref{isorep} (ii), and the field $\K$ is either $\R$ or $\H$. Consequently,
\beq\label{mm2}\mm\otimes_\R \CM=\mu_+^\C\otimes_\C\lambda_+^\C\oplus \mu_-^\C\otimes_\C\lambda_-^\C.\eeq

Let us denote by $\A:=\{\alpha_1,\ldots,\alpha_p\}\subset\tt^*$ the weights of the representation $\l^\C$, defined when
$r$ is not a multiple of $4$.
For $r=2q+1$ $\l^\C$ is self-dual, so $\A=-\A$. Moreover, $\K=\H$ if $q\equiv 1$ or $2\mod 4$, so $p=\#\A$ is even,
whereas for $q\equiv 0$ or $3\mod 4$ $p$ might be odd, {\em i.e.} one of the vectors $\alpha_i$ may vanish.

For $r=2q$ with $q$ even, we denote by $\A:=\{\alpha_1,\ldots,\alpha_p\}$ and $\G:=\{\gamma_1,\ldots,\gamma_{p'}\}$
the weights of the representations $\l_\pm^\C$. Since they are both self-dual we have $\A=-\A$ and $\G=-\G$ and 
we note that $\K=\H$ for $q\equiv 2\mod4$, whence $p$ and $p'$ are even in this case.

Recall now that for $r=2q+1$, the weights of the complex spin representation $\S_r^\C$
are
$$\W(\S_r^\C)=\left\{\sum_{j=1}^q\e_j e_j,\ \e_j=\pm1\right\},$$
where $\{e_j\}$ is some orthonormal basis of the dual of some Cartan subalgebra of
$\so(2q+1)$. Similarly, if $r=2q$ with $q$ odd, the weights of the complex spin representation $\S_r^\C$ are
$$\W(\S_r^\C)=\left\{\sum_{j=1}^q\e_j e_j,\ \e_j=\pm1,\overset{q}{\underset{j=1}
{\prod}}\e_j=1\right\},$$
and for $r=2q$ with $q$ even,
the weights of the complex half-spin representations $(\S_r^\pm)^\C$ are
$$\W((\S_r^+)^\C)=\left\{\sum_{j=1}^q\e_j e_j,\ \e_j=\pm1,\overset{q}{\underset{j=1}
{\prod}}\e_j=1\right\},$$
$$\W((\S_r^-)^\C)=\left\{\sum_{j=1}^q\e_j e_j,\ \e_j=\pm1,\overset{q}{\underset{j=1}
{\prod}}\e_j=-1\right\}.$$

We denote by $\beta_j\in\tt^*$ the pull-back through $\mu_*$ of the vectors
$\frac{1}{2}e_j$, for $j=1,\dots,q$. Since $\mu=\xi\circ\rho_*$ (and $\mu_\pm=\xi^\pm\circ\rho_*$
for $r$ multiple of $4$), the above relations give directly the weights of $\mu^\C$ or
$\mu_\pm^\C$ as linear combinations of the vectors $\beta_j$.
Taking into account Lemma \ref{adm}, Lemma \ref{isorep}, \eqref{mm1}-\eqref{mm2} and 
the previous discussion, we obtain the following description of the weights of the isotropy representation
of a homogeneous Clifford structure:

\begin{Proposition}\label{wisotrop}
If there exists a homogeneous Clifford structure of rank $r$ on
a compact homogeneous space $G/H$ with $\rk(G)=\rk(H)$,
then the set $\mathcal{W}:=\mathcal{W}(\mathfrak{m})$ of weights of the isotropy representation
is an admissible subsystem of roots of $\mathcal{R}(\mathfrak{g})$
and is of one of the following types:
\begin{enumerate}
\item[(I)] If $r=2q+1$, then there exists $\A:=\{\alpha_1,\ldots,\alpha_p\}\subset\tt^*$ with $\A=-\A$ such that
  $\mathcal{W}=\A+\{\overset{q}{\underset{j=1}
{\sum}}\varepsilon_j\beta_j\}_{\varepsilon\in\mathcal{E}_q}$ and $\#\mathcal{W}=p\cdot 2^{q}$. 
Moreover, if $q\equiv 1$ or $2\mod 4$ then $p$ is even, so $\alpha_i\ne 0$ for all $i$.
\item[(II)] If $r=2q$ with $q$ odd, then
  $\mathcal{W}=\{(\overset{q}{\underset{j=1}
{\prod}}\varepsilon_j)\alpha_i+\overset{q}
{\underset{j=1}{\sum}}\varepsilon_j\beta_j \}_{ i=\overline{1,p},
  \varepsilon\in\mathcal{E}_q}$ and $\#\mathcal{W}=p\cdot 2^{q}$.
\item[(III)] If $r=2q$ with $q\equiv 2\mod 4$, then there exist $\A:=\{\alpha_1,\ldots,\alpha_p\}$ and $\G:=\{\gamma_1,\ldots,\gamma_{p'}\}$
in $\tt^*$ with $\A=-\A$ and $\G=-\G$ such that
  $$\mathcal{W}=\A+\left\{\overset{q}{\underset{j=1}{\sum}}
\varepsilon_j\beta_j |\, \overset{q}{\underset{j=1}{\prod}}
\varepsilon_j=1\right\}_{\varepsilon\in\mathcal{E}_q}
\bigcup\G+\left\{\overset{q}{\underset{j=1}{\sum}}
\varepsilon_j\beta_j |\, \overset{q}{\underset{j=1}
{\prod}}\varepsilon_j=-1\right\}_{\varepsilon\in\mathcal{E}_q}$$
and $\#\mathcal{W}=(p+p')\cdot 2^{q-1}$. In this case one of $p$ or $p'$ might vanish, but $p$ and $p'$ are even, so
the vectors $\alpha_i$ and $\gamma_i$ are all non-zero.
\item[(IV)] If $r=2q$ with $q\equiv 0\mod 4$ (in this case the
  semi-spinorial representation is real), then there exist $\A:=\{\alpha_1,\ldots,\alpha_p\}$ and $\G:=\{\gamma_1,\ldots,\gamma_{p'}\}$
in $\tt^*$ with $\A=-\A$ and $\G=-\G$ such that
  $$\mathcal{W}=\A+\left\{\overset{q}{\underset{j=1}{\sum}}
\varepsilon_j\beta_j |\, \overset{q}{\underset{j=1}{\prod}}
\varepsilon_j=1\right\}_{\varepsilon\in\mathcal{E}_q}
\bigcup\G+\left\{\overset{q}{\underset{j=1}{\sum}}\varepsilon_j\beta_j
|\, \overset{q}{\underset{j=1}{\prod}}\varepsilon_j=-1\right\}_{\varepsilon\in\mathcal{E}_q}$$
and  $\#\mathcal{W}=(p+p')\cdot 2^{q-1}$. In this case one of $p$ or $p'$ might vanish, as well as one of the vectors
$\alpha_i$ or $\gamma_i$.
\end{enumerate}
\end{Proposition}

In order to describe the homogeneous Clifford structures we shall now obtain
by purely algebraic arguments
several restrictions on the possible sets of weights of the isotropy representation given
by Proposition~\ref{wisotrop}.

\begin{Proposition}\label{linalg}
Let $\A:=\{\alpha_1,\ldots,\alpha_p\}$ and $\B:=\{\beta_1,\dots,\beta_q\}$ be subsets in a Euclidean space
$(V,\<\cdot,\cdot>)$ with $\beta_j\neq 0, j=\overline{1,q}$. The following restrictions
for $q$ hold:
\begin{enumerate}
\item[(I)] If $\A=-\A$ and
  $\mathcal{P}_1:=\A+\{\overset{q}
{\underset{j=1}{\sum}}\varepsilon_j\beta_j\}_{\varepsilon\in\mathcal{E}_q}$ is a subsystem of roots, then
$q\leq4$.  Moreover, if $q=4$, then $\alpha_i=0$ for all $1\leq i\leq p$.

\item[(II)] If $q$ is odd and
  $\mathcal{P}_2:=\{(\overset{q}{\underset{j=1}{\prod}}\varepsilon_j)
\alpha_i+\overset{q}{\underset{j=1}{\sum}}\varepsilon_j\beta_j \}_{
  i=\overline{1,p},  \varepsilon\in\mathcal{E}_q}$ is a subsystem of
roots, then $q\leq7$.  Moreover, if $q=5$ or $q=7$, then $\alpha_i\neq 0$ for all $1\leq i\leq p$.

\item[(III)-(IV)] If $q$ is even, $\A=-\A$ and
$$\mathcal{P}_3:=\A+\left\{\overset{q}{\underset{j=1}
{\sum}}\varepsilon_j\beta_j
|\, \overset{q}{\underset{j=1}{\prod}}\varepsilon_j=1\right\}_{\varepsilon\in\mathcal{E}_q}$$ or
$$\mathcal{P}_4:=\A+\left\{\overset{q}{\underset{j=1}{\sum}}\varepsilon_j\beta_j
|\, \overset{q}{\underset{j=1}{\prod}}\varepsilon_j=-1\right\}_{\varepsilon\in\mathcal{E}_q}$$ is a subsystem of
roots, then $q\leq8$. Moreover, if $q=8$, then $\alpha_i=0$ for all $1\leq i\leq p$. 
Thus, if there exists some $\alpha_i\neq 0$, it follows that $q\leq 6$.
\end{enumerate}
\end{Proposition}

\begin{proof}
(I) If there exists $i\in\{1,\dots,p\}$ such that $\alpha_i\neq 0$,
let $\beta_0:=\alpha_i$ and
$\beta:=\overset{q}{\underset{j=0}{\sum}}\beta_j$. By changing the
signs if necessary, we may assume without loss of generality that
$\beta$ has the largest norm among all elements of $\mathcal{P}_1$ and
$\norm(\beta)=1$. From \eqref{prodscal} it follows that 
\[\<\beta,\beta-2\beta_j>\in\left\{0,\pm \frac{1}{2}\right\}, j=0,\dots,q.\]
We then have $\frac{1}{2}(q+1)\geq
\underset{j=0}{\overset{q}{\sum}}\<\beta,\beta-2\beta_j>=q-1$, 
which shows that $q\leq 3$. If $\alpha_i=0$ for all
$i\in\{1,\dots,p\}$, it follows by the same argument that $q\leq 4$. 

(II) If there exists $i\in\{1,\dots,p\}$ such that $\alpha_i=0$, then
$\{\overset{q}{\underset{j=1}{\sum}}\varepsilon_j\beta_j\}
_{\varepsilon\in\mathcal{E}_q}\subset\mathcal{P}_2$
and it follows from (I) that $q\leq 4$. 
In particular, this shows that if $q\in\{5,7\}$, then $\alpha_i\neq0$,
for all $i=1,\dots,p$.

Otherwise, if  $\alpha_i\neq 0$ for all $i=1,\dots,p$, then by
denoting $\beta_0:=\alpha_1$, we have 
$\{(\overset{q}{\underset{j=1}{\prod}}\varepsilon_j)
\alpha_1+\overset{q}{\underset{j=1}{\sum}}\varepsilon_j\beta_j\}_
{\varepsilon\in\mathcal{E}_q}=\{\overset{q}
{\underset{j=0}{\sum}}\varepsilon_j\beta_j
|\, \overset{q}{\underset{j=0}{\prod}}\varepsilon_j=1\}_  
{\varepsilon\in\mathcal{E}_{q+1}}\subset\mathcal{P}_2$. This subset is
of the same type as those considered in (III)--(IV) with $q+1$ even and
with all $\alpha_i=0$. It then follows from (III)--(IV) that $q+1\leq 8$, so
$q\leq 7$.

(III)--(IV): If we denote by $\beta'_j:=\beta_{2j-1}+\beta_{2j}$, for $j=1,\dots,\frac{q}{2}$,
then
$\A+\{\overset{q/2}{\underset{j=1}{\sum}}
\varepsilon_j\beta'_j\}_{\varepsilon\in\mathcal{E}_{q/2}}\subset\mathcal{P}_3$ is a subsystem of roots. It then
follows from (I) that $q\leq 8$ and the equality is attained only if
$\alpha_i=0$, for all $i=1,\dots, p$. The same argument holds for
$\mathcal{P}_4$ if we choose $\beta'_1=\beta_1-\beta_2$ and
$\beta'_j:=\beta_{2j-1}+\beta_{2j}$, for $
j=2,\dots,\frac{q}{2}$.   
\end{proof}

We now give a more precise description of the subsystems of roots that may occur in the limiting cases of
Proposition~\ref{linalg}. Namely, we determine all the possible scalar products between the roots.

\begin{Lemma}\label{typeI}
$(a)$ Let
$\mathcal{P}:=\{\overset{q}{\underset{j=1}{\sum}}
\varepsilon_j\beta_j\}_{\varepsilon\in\mathcal{E}_q}$ be a subsystem of roots
   with $\#\mathcal{P}=2^{q}$.
\begin{enumerate}
\item[(i)] If $q=4$, then the Gram matrix of scalar products $(\<\beta_i,\beta_j>)_{ij}$ is 
(up to a permutation of the subscripts and sign changes)
one of the following:
\begin{equation}\label{m4}
\frac{1}{4}\id_4 \text{ or } M_0:=\left(\begin{array}{cccc}
\frac{1}{4} & 0 & 0 & 0 \\
0 & \frac{1}{8} & \frac{1}{16} & \frac{1}{16} \\
0 & \frac{1}{16} & \frac{1}{8} & \frac{1}{16} \\
0 & \frac{1}{16} & \frac{1}{16} & \frac{1}{8} \\
\end{array}\right).
\end{equation}
Moreover, if $\mathcal{P}$ is admissible, then only the first case can occur, 
the Gram matrix is $(\<\beta_i,\beta_j>)_{ij}=\frac{1}{4}\id_4$ and $\overline{\mathcal{P}}=\mathcal{R}(\so(8))$,
$\overline{\mathcal{P}}\setminus\mathcal{P}=\mathcal{R}(\so(4)\oplus\so(4))$.
\item[(ii)] If $q=3$, then the Gram matrix of scalar products $(\<\beta_i,\beta_j>)_{ij}$ is 
(up to a permutation of the subscripts and sign changes) one of the following:
\begin{equation}\label{m3}
M_1:=\left(\begin{array}{ccc}
\frac{1}{2} & 0 & 0\\
0 & \frac{1}{4} & 0 \\
0 & 0 & \frac{1}{4} \\
\end{array}\right)\text{ or } M_2:=\left(\begin{array}{ccc}
\frac{1}{3} & 0 & \frac{1}{6} \\
0 & \frac{1}{4} & 0 \\
\frac{1}{6} & 0 & \frac{1}{12} \\
\end{array}\right) \text{ or } M_3:=\left(\begin{array}{ccc}
\frac{3}{8} & \frac{1}{16} & \frac{1}{16} \\
\frac{1}{16} & \frac{1}{8} & \frac{1}{16} \\
\frac{1}{16} & \frac{1}{16} & \frac{1}{8} \\
\end{array}\right).
\end{equation}
Moreover, if $\mathcal{P}$ is admissible, then only the first two cases can occur. 
For the Gram matrix $(\<\beta_i,\beta_j>)_{ij}=M_1$ the subsystems of roots are $\overline{\mathcal{P}}=\mathcal{R}(\so(6))$ and
$\overline{\mathcal{P}}\setminus\mathcal{P}=\mathcal{R}(\so(4))$ and for $(\<\beta_i,\beta_j>)_{ij}=M_2$,
$\overline{\mathcal{P}}=\mathcal{R}(\mathfrak{g}_2)$ and
$\overline{\mathcal{P}}\setminus\mathcal{P}=\mathcal{R}(\so(4))$.
\end{enumerate}
$(b)$ Let
$\mathcal{P}:=\{\overset{q}{\underset{j=1}{\sum}}
\varepsilon_j\beta_j |\, \overset{q}{\underset{j=1}{\prod}}
\varepsilon_j=1\}_{\varepsilon\in\mathcal{E}_q}$ be an admissible subsystem of roots
  with $\#\mathcal{P}=2^{q}$. If $q=8$, then
  the Gram matrix $(\<\beta_i,\beta_j>)_{ij}$ is 
  (up to a permutation of the subscripts and sign changes)
  equal to $\frac{1}{8}\id_8$.
\end{Lemma}

\begin{proof}
(i) As in the proof of Proposition~\ref{linalg},
we denote by $\beta:=\overset{q}{\underset{j=0}{\sum}}\beta_j$
and, up to sign changes, we may assume that
$\beta$ has the largest norm among all elements of $\mathcal{P}$ and
that this norm is equal to $1$.
We consider the Gram matrix of scalar products $(\<\beta_i,\beta_j>)_{ij}$.
Since $\<\beta_j,\beta>=\frac{1}{4}$ for all $j=\overline{1,4}$, 
the sum of the elements of each of its lines is $\frac{1}{4}$.

Since $\norm(\beta)=1$ is the largest norm of the roots in
$\mathcal{P}$, it follows that the square norms of the other roots may
take the following values: $\{1,\frac{1}{2}, \frac{1}{3}\}$, so
that
\begin{equation}\label{norms}
\norm(\beta-2\beta_i)=4\norm(\beta_i) \Rightarrow
\norm(\beta_i)\in\left\{\frac{1}{4}, \frac{1}{8},\frac{1}{12}\right\}, \text{ for all }
1\leq i\leq 4.
\end{equation} 

Let $i,j\in\{1,\dots,4\}$, $i\neq j$ and assume that $\norm(\beta_i)\geq
\norm(\beta_j)$.
As $\<\beta-2\beta_i,\beta-2\beta_j>=4\<\beta_i,\beta_j>$, it follows by \eqref{prodscal}
that $\<\beta_i,\beta_j>\in\{0,\pm\frac{1}{2}\norm(\beta_i)\}$. 

The case $\<\beta_i,\beta_j>=-\frac{1}{2}\norm(\beta_i)$ cannot
occur, because it leads to the following contradiction:
\[0<\norm(\beta-2\beta_i-2\beta_j)=4\norm(\beta_i+\beta_j)-1=
4\norm(\beta_j)-1\leq
0.\] 

Assume that there exists $i,j\in\{1,\dots,4\}$, $i\neq j$, such that
 $\<\beta_i,\beta_j>=\frac{1}{2}\norm(\beta_i)$. It then follows
$\norm(\beta_i+\beta_j)=2\norm(\beta_i)+\norm(\beta_j)\in
\{\frac{1}{2},\frac{3}{8},\frac{1}{3}\}$,
which combined with the restrictions \eqref{norms} yield the following
possible values: either $\norm(\beta_i)=\norm(\beta_j)=\frac{1}{8}$
or $\norm(\beta_i)=\frac{1}{8},
\norm(\beta_j)=\frac{1}{12}$.
In both cases $\<\beta_i,\beta_j>=\frac{1}{16}$.
Hence, all non-diagonal entries are either $0$ or $\frac{1}{16}$
and since the sum of the elements of each line is equal to $\frac{1}{4}$, the case
$\norm(\beta_j)=\frac{1}{12}$ is excluded.
It follows that each line of the Gram matrix $(\<\beta_i,\beta_j>)_{ij}$ either
has all non-diagonal entries equal to $0$ and the diagonal term is
$\frac{1}{4}$ or two of them are $\frac{1}{16}$, one is $0$ and the
diagonal term is $\frac{1}{8}$. In particular,
$\norm(\beta_i)\in\{\frac{1}{4},\frac{1}{8}\}$, for all $1\leq i\leq 4$.
Up to a permutation of the subscripts we may assume that
$\norm(\beta_i)\geq\norm(\beta_j)$ for all $1\leq i<j\leq 4$.
The following two cases may occur: either
$\norm(\beta_1)=\norm(\beta_2)=\frac{1}{4}$ or
$\norm(\beta_1)=\frac{1}{4}$ and
  $\norm(\beta_2)=\frac{1}{8}$.
  
In the first case at most one non-diagonal term may be $\frac{1}{16}$, 
showing that they must all be equal to $0$. This contradicts $\<\beta_i,\beta_j>=\frac{1}{2}\norm(\beta_i)\neq 0$.
In the second case, since we have one
  non-diagonal term equal to $\frac{1}{16}$, we must have another one
  and then $\norm(\beta_3)=\norm(\beta_4)=\frac{1}{8}$. Hence, the
  corresponding Gram matrix is equal to $M_0$ defined by \eqref{m4}.

The subsystem of roots corresponding to the matrix $M_0$ is not
admissible, since the necessary criterion given by
Lemma~\ref{admis} is not fulfilled: the set
$\{\beta_i\}_{i=\overline{1,4}}$ is linearly independent (since the
Gram matrix is invertible) and $\beta$, $\beta-2\beta_2$ are two roots
of the subsystem, which have different norms
$\norm(\beta-2\beta_2)=\frac{1}{2}=\frac{1}{2}\norm(\beta)$ and are
not orthogonal: $\<\beta,\beta-2\beta_2>=\frac{1}{2}$. 

The only case left is
when $\{\beta_i\}_{i=\overline{1,4}}$ is an orthonormal system. In
this case the new roots obtained by considering all possible
reflections are the roots $\{\pm2\beta_i\}_{i=\overline{1,4}}$,
which are of the same norm and orthogonal to each other and thus build 
the system of roots of $\su(2)\oplus \su(2)\oplus \su(2)\oplus\su(2)$.
Then the minimal set of roots is
$\overline{\mathcal{P}}=\{\pm\overset{4}{\underset{j=1}{\sum}}
\varepsilon_j\beta_j\}\cup\{\pm2\beta_i\}_{i=\overline{1,4}}$,
which is the system of roots of $\so(8)$, proving (i). 

(ii) If $q=3$, we assume as above that $\beta:=\beta_1+\beta_2+\beta_3$ is the
element of $\mathcal{P}$ of maximal norm and $\norm(\beta)=1$. From
\eqref{prodscal} it follows that $\< \beta,
\beta-2\beta_i>\in\{0,\pm\frac{1}{2}\}$, for all $i=\overline{1,3}$,
yielding that $\< \beta, \beta_i>\in\{\frac{1}{2}, \frac{1}{4},
\frac{3}{4}\}$. Since $1=\norm(\beta)=\sum_{i=1}^{3}\<\beta_i,\beta>$,
the only possible values (up to a permutation of the subscripts) are
$\<\beta_1,\beta>=\frac{1}{2}$ and
$\<\beta_2,\beta>=\<\beta_3,\beta>=\frac{1}{4}$.  
Then, from $\norm(\beta-2\beta_i) \in\{1, \frac{1}{2}, \frac{1}{3}\}$,
it follows that
\begin{equation}\label{norms1}
\norm(\beta_1)\in\left\{\frac{1}{2}, \frac{3}{8}, \frac{1}{3}\right\}, \quad
\norm(\beta_2), \norm(\beta_3)\in\left\{\frac{1}{4}, \frac{1}{8},
\frac{1}{12}\right\}.
\end{equation}
Since $\< \beta, \beta-2\beta_1>=0$ and $\< \beta, \beta-2\beta_2>=\<
\beta, \beta-2\beta_3>=\frac{1}{2}$, we obtain the following
expressions for the scalar products:
\begin{align}
2\< \beta_2, \beta_3>&=\norm(\beta_1)-\norm(\beta_2)-\norm(\beta_3),
\label{prodscals1} \\ 
2\< \beta_1,
\beta_2>&=\frac{1}{2}+\norm(\beta_3)-\norm(\beta_1)-\norm(\beta_2),
\label{prodscals2} \\
2\< \beta_1,
\beta_3>&=\frac{1}{2}+\norm(\beta_2)-\norm(\beta_1)-\norm(\beta_3).
\label{prodscals3} 
\end{align}

The other conditions for the scalar products of roots in $\mathcal{P}$
obtained from \eqref{prodscal} are:
$\<\beta-2\beta_i,
\beta-2\beta_j>\in\{0,\pm\frac{1}{2}\max(\norm(\beta-2\beta_i),
\norm(\beta-2\beta_j))\}$, for all $1\leq i <j\leq 3$, which imply 
\begin{align}
\<\beta_2,\beta_3>&\in\left\{0,\pm\frac{1}{2}\max(\norm(\beta_2),
\norm(\beta_3))\right\}, \label{cdtscal1} \\
\<\beta_1,\beta_i>&\in\left\{\frac{1}{8},
\frac{1}{8}\pm\frac{1}{2}\max(\norm(\beta_1)-\frac{1}{4},
\norm(\beta_i))\right\}, i=2,3. \label{cdtscal2} 
\end{align}

We may assume (up to a permutation) that $\norm(\beta_2)\geq
\norm(\beta_3)$. By substituting \eqref{prodscals1}--\eqref{prodscals3}
in \eqref{cdtscal1}--\eqref{cdtscal2} we obtain the following
conditions:
$\norm(\beta_1)-\norm(\beta_2)-\norm(\beta_3)\in\{0,\pm\norm(\beta_2)\}$
and
$\norm(\beta_2)-\norm(\beta_1)-\norm(\beta_3)+\frac{1}{4}\in
\{0,\pm\norm(\beta_3)\}$,
which together with the restrictions  \eqref{norms1} for the norms
yield the following possible values:  
\[(\norm(\beta_1), \norm(\beta_2),
\norm(\beta_3))\in\left\{\left(\frac{1}{2},\frac{1}{4},
\frac{1}{4}\right),\left(\frac{1}{3},\frac{1}{4},
\frac{1}{12}\right),\left(\frac{3}{8},\frac{1}{8},
\frac{1}{8}\right)\right\}.\]  

We thus obtain that the Gram matrix $(\<\beta_i,\beta_j>)_{ij}$
must be equal to one of the three matrices $M_i$, $i=\overline{1,3}$, defined by \eqref{m3}.

By Lemma~\ref{admis}, the subsystem of roots corresponding to $M_3$ is not
admissible, since the set
$\{\beta_i\}_{i=\overline{1,3}}$ is linearly independent (its Gram
matrix is invertible) and there exist two roots $\beta$,
$\beta-2\beta_2$ of different norms
$\norm(\beta-2\beta_2)=\frac{1}{2}=\frac{1}{2}\norm(\beta)$, which are
not orthogonal: $\<\beta,\beta-2\beta_2>=\frac{1}{2}$. 

The subsystem of roots corresponding to the matrix $M_1$ is admissible and in
this case the minimal set of roots containing $\mathcal{P}$ is
obtained by adjoining the roots $\pm2\beta_2$ and $\pm2\beta_3$, which
also have norm equal to $1$, showing that
$\overline{\mathcal{P}}=\mathcal{R}(\so(6))$ and
$\overline{\mathcal{P}}\setminus \mathcal{P}=\mathcal{R}(\so(4))$.

For the Gram matrix $M_2$, it follows that $\beta_1=2\beta_3$ and
$\{\beta_1,\beta_2\}$ are linearly independent. The roots in
$\mathcal{P}$ have the following norms:
$\norm(\beta)=\norm(\beta-2\beta_2)=1$ and
$\norm(\beta-2\beta_1)=\norm(\beta-\beta_1)=\frac{1}{3}$ and the roots
added by all possible reflections in order to obtain the minimal system of roots
$\overline{\mathcal{P}}$ are $\{\pm\beta_1, \pm2\beta_2\}$ with
$\norm(\beta_1)=\frac{1}{3}$, $\norm(2\beta_2)=1$ and $\<\beta_1,2\beta_2>=0$.
It thus follows that
$\overline{\mathcal{P}}\setminus\mathcal{P}=\mathcal{R}(\su(2)\oplus
\su(2))$ and $\overline{\mathcal{P}}=\mathcal{R}(\mathfrak{g}_2)$ (cf. \cite[p. 32]{a}). 

(b) If we denote by
$\beta'_j:=\beta_{2j-1}+\beta_{2j}$, for $j=1,\dots,4$, then
$\{\overset{4}{\underset{j=1}{\sum}}\varepsilon_j\beta'_j\}_{
\varepsilon\in\mathcal{E}_{4}}\subset\mathcal{P}$
is a subsystem of roots. From (i) it follows that the Gram
matrix  $(\<\beta'_i, \beta'_j>)_{i,j}$ is (up to rescaling, reordering and sign change of
the vectors $\beta'_j$) either $\frac{1}{4}\id$ or the matrix $M_0$ defined by \eqref{m4}. 

We claim that the second case cannot occur. Indeed, if this were the case, then
$\norm(\beta_1+\beta_2)=\frac14$,
$\norm(\beta_{2j-1}+\beta_{2j})=\frac18$ for $j=2,3,4$, and
\beq \label{prs}\<\beta_3+\beta_4,\beta_5+\beta_6>=\frac1{16}.\eeq
For every $j$ and $k$ with $3\le j<k\le 8$, there exists $l\in\{2,3,4\}$
such that $\{2l-1,2l\}\cap\{j,k\}=\emptyset$. Let $\{s,t\}$ denote the
complement of $\{1,2,j,k,2l-1,2l\}$ in $\{1,\ldots,8\}$.
The Gram matrix
of $\{\beta_1+\beta_2,
\beta_{2l-1}+\beta_{2l},\beta_j-\beta_k,\beta_s-\beta_t\}$ has at least two
different values on the diagonal. Again from (i), it follows that the remaining diagonal terms
$\norm(\beta_j-\beta_k)$ and $\norm(\beta_s-\beta_t)$ must both be equal to
$\frac18$. Thus,
$\norm(\beta_j-\beta_k)=\frac18$ for all $3\le j<k\le 8$.
By the same argument we also obtain
$\norm(\beta_j+\beta_k)=\frac18$ for all $3\le j<k\le 8$. Thus,
$\<\beta_j,\beta_k>=0$ for all $3\le j<k\le 8$, contradicting
\eqref{prs}.

This shows that the vectors $\beta'_j$ are mutually orthogonal. Applying
this to different partitions of the set $\{1,\ldots,8\}$ into four pairs
we get that $\beta_j+\beta_k$ is orthogonal to $\beta_s+\beta_t$ for
all mutually distinct subscripts $j,k,s$ and $t$. This clearly implies that
$\<\beta_i, \beta_j>=0$ for all $i\neq j$. It then also follows that
$\norm(\beta_j)=\frac{1}{8}$, for $j=1,\dots,8$, proving (b).
\end{proof}

\section{Homogeneous Clifford Structures of high rank}

A direct consequence of Propositions~\ref{wisotrop} and \ref{linalg} is the
following upper bound for the rank of a homogeneous Clifford structure:

\begin{Theorem}\label{upperbd}
The rank $r$ of any even homogeneous Clifford structure on a homogeneous compact
manifold $G/H$ of non-vanishing Euler characteristic is less or equal to $16$.
More precisely, the following restrictions hold for the rank depending on its $2$-valuation:
\begin{enumerate}
\item[(I)] If $r$ is odd, then $r\in\{3,5,7,9\}$.
\item[(II)] If $r\equiv 2\mod 4$, then $r\in\{2,6,10\}$.
\item[(III)] If $r\equiv 4\mod 8$, then $r\in\{4,12\}$.
\item[(IV)] If $r\equiv 0\mod 8$, then $r\in\{8,16\}$.
\end{enumerate}
\end{Theorem}
\begin{proof}
We only need to show that in case (II), the rank $r$ is strictly less than $14$. 
This will be done in the proof of Theorem \ref{h-types} (II).
\end{proof}

We further describe the manifolds which occur in the limiting cases for the upper bounds in
Theorem~\ref{upperbd}. 

\begin{Theorem}\label{h-types}
The maximal rank $r$ of an even homogeneous Clifford structure for each of the types (I)-(IV)
and the corresponding compact homogeneous manifolds $M=G/H$ (with $\rk(G)=\rk(H)$)
carrying such a structure are the following:
\begin{enumerate}
\item[(I)] $r=9$ and $M$ is the Cayley projective space
$\mathbb{O}\mathbb{P}^2=\mathrm{F}_4/\mathrm{Spin}(9)$.
\item[(II)] $r=10$ and $M=(\mathbb{C}\otimes\mathbb{O})\mathbb{P}^2=
\mathrm{E}_6/(\mathrm{Spin}(10)\times\mathrm{U}(1)/\ZM_4)$.
\item[(III)] $r=12$ and $M=(\mathbb{H}\otimes\mathbb{O})\mathbb{P}^2=
\mathrm{E}_7/\mathrm{Spin}(12)\cdot\mathrm{SU}(2)$.
\item[(IV)] $r=16$ and $M=(\mathbb{O}\otimes\mathbb{O})\mathbb{P}^2=
\mathrm{E}_8/(\mathrm{Spin}(16)/\mathbb{Z}_2)$.
\end{enumerate}
\end{Theorem}

\begin{proof}
Let $M=G/H$ ($\rk(G)=\rk(H)$) carry a homogeneous Clifford structure of rank $r$.

(I) If $r$ is odd, $r=2q+1$, then it follows from Theorem~\ref{upperbd} (I) that $r\leq 9$.

By Proposition~\ref{wisotrop} (I) and Lemma~\ref{typeI} (i), the set $\mathcal{W}:=\mathcal{W}(\mathfrak{m})$ 
of weights of the isotropy representation is 
$\mathcal{W}=\{\overset{q}{\underset{j=1}
{\sum}}\varepsilon_j\beta_j\}_{\varepsilon\in\mathcal{E}_q}$ with $\#\mathcal{W}=2^{q}$ 
and $\mathcal{R}(\mathfrak{so}(8))\subseteq\mathcal{R}(\mathfrak{g})$.
In particular the representation $\lambda$ is trivial, so $\mathfrak{h}=\mathfrak{h}_2$ and $\#\mathcal{R}(\mathfrak{g})\geq 24$.

Since $\rho_*:\mathfrak{h}_2\to\so(9)$ is
injective and $\mathfrak{h}=\mathfrak{h}_2$,
it follows that $\rk(\mathfrak{h})\leq 4$.

If $\rk(\mathfrak{h})\leq 3$, then $\#\mathcal{R}(\mathfrak{g})\leq 18$
by a direct check in the list of Lie algebras of rank $3$. This contradicts the fact that
$\#\mathcal{R}(\mathfrak{g})\geq 24$.

Thus, $\rk(\mathfrak{h})=4$ and $\rho_*$ is a
bijection when restricted to a Cartan subalgebra of
$\mathfrak{h}$. On the other hand, from Lemma~\ref{typeI} (i), it
also follows that the Gram matrix $(\<\beta_i,\beta_j>)_{ij}$
is equal to $\frac{1}{4}\id_4$. The new roots obtained
by reflections given by \eqref{newroots} are the following:
$\pm 2\beta_i=\pm e_i\circ\rho_*\in\mathcal{R}(\mathfrak{h})$, for all
$i=\overline{1,4}$. As $\rk(\mathfrak{h})=4=\rk(\so(9))$, we may apply
Lemma~\ref{sumroots} for $\mathfrak{h}\subseteq\mathfrak{so}(9)$
and get that $\{\pm e_i\pm e_j \,|\, 1\leq
i<j\leq 4\}$, which are roots of $\so(9)$, are also roots of
$\mathfrak{h}$. Thus, $\mathfrak{h}=\so(9)$.
The Lie algebra $\mathfrak{g}$, whose system of roots is obtained by
joining the system of roots of $\so(9)$ with the weights of the
spinorial representation of $\spin(9)$, is then exactly $\mathfrak{f}_4$
(cf. \cite[p. 55]{a}).
We note that we cannot extend $\mathfrak{f}_4$, since there is no other
larger Lie algebra of the same rank.
Using the fact that the closed subgroup of $\F_4$ corresponding to the above embedding of $\so(9)$ 
in $\mathfrak{f}_4$ is $\Spin(9)$ (cf. \cite[Thm. 6.1]{a}), we deduce that the only homogeneous manifold carrying a
homogeneous Clifford structure of rank $9$ is
the Cayley projective space $\mathbb{O}\mathbb{P}^2=\mathrm{F}_4/\mathrm{Spin}(9)$.

(II) By Theorem~\ref{upperbd} (II), $r\leq 14$. We first show that
there exists no homogeneous Clifford structure of rank $r=14$.

Let $r=2q=14$. In this case, by Proposition~\ref{wisotrop} (II), the set of weights of
the isotropy representation is 
$\mathcal{W}:=\mathcal{W}(\mathfrak{m})=\left\{(\overset{7}{\underset{j=1}{\prod}}\varepsilon_j)
\alpha_i+\overset{7}{\underset{j=1}{\sum}}\varepsilon_j\beta_j \right\}_{
  i=\overline{1,p},  \varepsilon\in\mathcal{E}_7}$ with $\#\mathcal{W}=p\cdot 2^{7}$.
Proposition~\ref{linalg} (II) yields that $\alpha_i\neq 0$, for all $1\leq i\leq p$.

We claim that the following inclusion holds:
$\mathcal{R}(\mathfrak{so}(16))\subseteq\mathcal{\overline{W}}\setminus \mathcal{W}$.
This can be seen as follows.
Denoting by $\beta_0:=\alpha_1$ and $\beta:=\beta_0+\dots+\beta_7$,
the set $\mathcal{W}$ contains the following subsystem of roots
$\{\overset{7}{\underset{j=0}{\sum}}\varepsilon_j\beta_j |\, 
\overset{7}{\underset{j=0}{\prod}}\varepsilon_j=1\}_{\varepsilon\in\mathcal{E}_8}$.
From Lemma~\ref{typeI} (b) it follows that the Gram matrix $(\<\beta_i, \beta_j>)_{i,j=\overline{0,7}}$ is equal to  $\frac{1}{8}\id_8$.
Then, for all $0\leq i< j\leq 7$ we have $\<\beta,\beta-2\beta_i-2\beta_j>=\frac{1}{2}$,
implying by \eqref{newroots} that there are new roots $\pm2(\beta_i+\beta_j)\in\mathcal{\overline{W}}\setminus \mathcal{W}$.
Similarly, for any $0\leq k\leq 7$ distinct from $i$ and $j$,  $\<\beta-2\beta_i-2\beta_k,\beta-2\beta_j-2\beta_k>=\frac{1}{2}$
yields the new roots $\pm2(\beta_i-\beta_j)\in\mathcal{\overline{W}}\setminus \mathcal{W}$.
It thus follows that
$\mathcal{R}(\mathfrak{so}(16))=\{\pm2(\beta_i\pm\beta_j)\,|\, 0\leq i\leq 7\}
\subseteq\mathcal{\overline{W}}\setminus \mathcal{W}$.

Since $\mathcal{R}(\mathfrak{so}(16))\subseteq\mathcal{\overline{W}}\setminus \mathcal{W}\subseteq\mathcal{R}(\mathfrak{h})$,
 it follows that $\mathfrak{so}(16)$ is a Lie subalgebra of $\mathfrak{h}$. Recall the splitting \eqref{hspl}:
$\mathfrak{h}=\mathfrak{h}_0\oplus \mathfrak{h}_{1}\oplus\mathfrak{h}_2$, 
where $\mathfrak{h}_0\oplus\mathfrak{h}_2\subseteq\mathfrak{so}(14)$. 
As $\mathfrak{so}(16)$ is a simple Lie algebra, it follows that
$\mathfrak{so}(16)\subseteq\mathfrak{h}_1$. In particular, this implies $p\geq 8$.

On the other hand, we show that $p=1$. Assume that $p\geq 2$. By Lemma~\ref{typeI} (b),
the Gram matrix of both subsystems of roots $\{\alpha_1,\beta_1,\dots,\beta_7\}$ 
and $\{\alpha_2,\beta_1,\dots,\beta_7\}$ is equal to $\frac{1}{8}\id_8$. 
Denoting by $a:=\<\alpha_1,\alpha_2>$, we obtain the following values for the scalar products
between the roots containing $\alpha_1$ and $\alpha_2$:
$\{a+\frac{7}{8}, a+\frac{3}{8}, a-\frac{1}{8}\}$. From \eqref{prodscal},
we know that these values must belong to $\{0,\pm\frac{1}{2}\}$. It then 
follows that the only possible value for $a=\<\alpha_1,\alpha_2>$ is $-\frac{3}{8}$. This leads to a contradiction
by computing the following norm: $\norm(\alpha_1+\alpha_2)=-\frac{1}{2}<0$.

Thus, the case $r=14$ is not possible.

Now, for rank $r=10=2q$, by Proposition~\ref{wisotrop} (II),
the set of weights of the isotropy representation is 
 $\mathcal{W}:=\mathcal{W}(\mathfrak{m})=\{(\overset{5}{\underset{j=1}
{\prod}}\varepsilon_j)\alpha_i+\overset{5}
{\underset{j=1}{\sum}}\varepsilon_j\beta_j \}_{ i=\overline{1,p},
\varepsilon\in\mathcal{E}_5}$ with $\#\mathcal{W}=p\cdot 2^{5}$.
From Proposition~\ref{linalg} (II), it then follows that $\alpha_i\neq 0$,
for all $1\leq i\leq p$.
We will further show that $p=1$ and
the following inclusions hold: 
$\mathcal{R}(\mathfrak{spin}(10)\oplus\mathfrak{u}(1))\subseteq\mathcal{\overline{W}}\setminus\mathcal{W}$,
$\mathcal{R}(\mathfrak{e}_{6})\subseteq \mathcal{\overline{W}}$.

Let $\beta_0:=\alpha_1$ and $\beta'_j:=\beta_{2j}+\beta_{2j+1}$, for $j=0,1,2$. Then
$\{\sum_{j=0}^{2}\varepsilon_j\beta'_j\}\subset\mathcal{W}$
is an admissible subsystem of roots. By Lemma~\ref{typeI} (ii), the Gram matrix
$B':=(\<\beta'_i, \beta'_j>)_{i,j=\overline{0,2}}$ is one of the three matrices
in \eqref{m3}. We may assume that $\beta'_0$ is of maximal norm
and equal to $1$, so that the possible Gram matrices $B'$ are normalized as follows:
\begin{equation}\label{m3n}
M_1:=\left(\begin{array}{ccc}
1 & 0 & 0\\
0 & \frac{1}{2} & 0 \\
0 & 0 & \frac{1}{2} \\
\end{array}\right), M_2:=\left(\begin{array}{ccc}
1 & 0 & \frac{1}{2} \\
0 & \frac{3}{4} & 0 \\
\frac{1}{2} & 0 & \frac{1}{4} \\
\end{array}\right), M_3:=\left(\begin{array}{ccc}
1 & \frac{1}{6} & \frac{1}{6} \\
\frac{1}{6} & \frac{1}{3} & \frac{1}{6} \\
\frac{1}{6} & \frac{1}{6} & \frac{1}{3} \\
\end{array}\right).
\end{equation}

We first note that the norm of $\beta:=\sum_{j=0}^{5}\beta_j$,
which is equal to the sum of all elements of the matrix $B'$,
may take the following values: $\norm(\beta)=2$ for $M_1$,
$\norm(\beta)=3$ for $M_2$, $\norm(\beta)=\frac{8}{3}$ for $M_3$.
We will show that the last two cases cannot occur.

Let us first assume that $B'=M_2$. In this case
$\beta'_0=2\beta'_2$, \emph{i.e.} $\beta_0+\beta_1=2(\beta_4+\beta_5)$. 
Considering now another pairing by permuting the subscripts
$2,3,4,5$, we get a Gram matrix which must also
be equal to $M_2$, since $\norm(\beta)$ does not change.
We may thus assume that $\beta_0+\beta_1=2(\beta_2+\beta_4)$ and is furthermore equal either
to $2(\beta_2+\beta_5)$ or to $2(\beta_3+\beta_4)$. 
In both cases it follows that there exists $i\neq j\in\{2,3,4,5\}$, such that
$\beta_i=\beta_j$. Then, for any $k\neq i,j$, the roots 
$\beta-2\beta_i-2\beta_k, \beta-2\beta_j-2\beta_k\in \mathcal{W}$ are equal,
which contradicts $\#\mathcal{W}=p\cdot 2^{5}$. Thus, $B'\neq M_2$.

Let us now assume that the Gram matrix $B'$ is equal to $M_3$. Then
$\norm(\beta)=\frac{8}{3}$ and since this is the maximal norm, it follows that
for any other possible pairing of the vectors $\beta_j$, the corresponding Gram matrix
is either $M_1$ or $M_3$ (because the sum of all elements of $M_2$ is $3>\frac{8}{3}$).

We consider as above other pairings by permuting the subscripts $\{2,3,4,5\}$.
Again by Lemma~\ref{typeI} (ii), it follows that the corresponding Gram matrix
is one of the matrices in \eqref{m3n} and, since  $\norm(\beta)$ does not change,
it must also be equal to $M_3$.
In particular, we have:
\begin{equation}\label{n+}
\norm(\beta_i+\beta_j) =\frac{1}{3},
\end{equation}
\begin{equation}\label{sp+}
\<\beta_i+\beta_j, \beta_k+\beta_l>=\frac{1}{6},
\end{equation}
for any permutation $(i,j,k,l)$ of $(2,3,4,5)$.

Consider now the following pairings of the vectors $\beta_j$:
\[\beta''_0:=\beta_0+\beta_1, \beta''_1:=\pm(\beta_i-\beta_j), \beta''_2:=\pm(\beta_k-\beta_l),\]
where $(i,j,k,l)$ is any permutation of $(2,3,4,5)$
and in each case the signs for $\beta''_1$ and $\beta''_2$ are chosen such that
$$\norm(\beta''_0+\beta''_1+\beta''_2)=\max\{\norm(\beta''_0\pm\beta''_1\pm\beta''_2)\}.$$
Then $\{\sum_{j=0}^{2}\varepsilon_j\beta''_j\}$ is a subsystem of roots of $\mathcal{W}$
and by the same argument as above its Gram matrix
$B'':=(\<\beta''_i, \beta''_j>)_{i,j=\overline{0,2}}$ is either $M_1$ or $M_3$.

Since $\norm(\beta''_0)=1$, it follows that in both cases
the norms of the other two vectors are equal: $\norm(\beta''_1)=\norm(\beta''_2)$,
\emph{i.e.} $\norm(\beta_i-\beta_j)=\norm(\beta_k-\beta_l)\in\{\frac{1}{2},\frac{1}{3}\}$.
By \eqref{n+}, we then obtain for any $i,j\in\{2,3,4,5\}$ that
$\<\beta_i,\beta_j>=\frac{1}{4}(\norm(\beta_i+\beta_j)-\norm(\beta_i-\beta_j))\in\{0,-\frac{1}{24}\}$.
It then follows that $\<\beta_i+\beta_j, \beta_k+\beta_l><0$,
for any permutation $(i,j,k,l)$ of $(2,3,4,5)$,
which contradicts \eqref{sp+}. Thus, $B'\neq M_3$.

The only possibility left is $B'=M_1$. Then $\norm(\beta)=2$ and since it is the element of maximal norm,
it follows that for any other pairing of the vectors $\beta_j$, the corresponding Gram matrix is also equal to $M_1$.
We then have $B'=B''=M_1$, which implies:
\begin{equation}\label{sp01}
\<\beta_0+\beta_1, \beta_i\pm\beta_j>=0,
\end{equation}
\begin{equation}\label{n+1}
\norm(\beta_i+\beta_j)=\norm(\beta_i-\beta_j)=\frac{1}{2},
\end{equation}
\begin{equation}\label{sp+1}
\<\beta_i+\beta_j, \beta_k+\beta_l>=\<\beta_i-\beta_j, \beta_k-\beta_l>=0,
\end{equation}
for any permutation $(i,j,k,l)$ of $(2,3,4,5)$.
From \eqref{sp01} it follows that $\<\beta_0+\beta_1, \beta_i>=0$  for all $2\leq i\leq 5$.
From \eqref{sp+1} it follows that $\<\beta_i,\beta_j>=0$, for $2\leq i<j\leq 5$ and then from \eqref{n+1}
we obtain $\norm(\beta_i)=\frac{1}{4}$, for $2\leq i\leq 5$.

For pairings of the following form $\beta_0-\beta_1$, $\beta_i-\beta_j$, $\beta_k+\beta_l$,
where again $(i,j,k,l)$ is a permutation of $(2,3,4,5)$, the corresponding Gram matrix must also be equal to $M_1$.
Since $\norm(\beta_i\pm\beta_j)=\frac{1}{4}$, for all $2\leq i<j\leq 5$, it follows that 
$\norm(\beta_0-\beta_1)=1$, so $\<\beta_0,\beta_1>=0$.
By the above argument applied to this pairing, it follows a similar relation, namely:
$\<\beta_0-\beta_1, \beta_i\pm\beta_j>=0$,
which together with \eqref{sp01} yields $\<\beta_0,\beta_i>=\<\beta_1,\beta_j>=0$, for $2\leq i\leq 5$.
Thus, the Gram matrix $(\<\beta_i,\beta_j>)_{0\leq i,j\leq 5}$ is diagonal with $\norm(\beta_0)+\norm(\beta_1)=1$
and $\norm(\beta_i)=\frac{1}{4}$, for $2\leq i\leq 5$.

The second element in the decreasing order of the set $\{\norm(\beta_i+\beta_j) \,|\, 0\leq i<j\leq 5\}$ must be,
up to a permutation of $0$ and $1$, of the form $\beta_0+\beta_k$, for some $2\leq k\leq 5$. 
By taking now, for instance, a pairing with first element equal to $\beta_0+\beta_k$,
it follows that its Gram matrix is equal to $M_1$ and by the same argument as above
 $\norm(\beta_1)=\frac{1}{4}$ and, consequently, $\norm(\beta_0)=\frac{3}{4}$.
Since we allowed permutations, we have to consider two cases:
$\norm(\alpha_1)\in\{\frac{3}{4}, \frac{1}{4}\}$.

If $\norm(\alpha_1)=\frac{1}{4}$, we may assume that $\norm(\beta_1)=\frac{3}{4}$ and 
$\norm(\beta_i)=\frac{1}{4}$, for all $2\leq i\leq 5$.
Since the Gram matrix is now completely known, we compute the scalar products between the roots in $\mathcal{W}$
and by \eqref{newroots} we obtain: $\{\pm2(\alpha_1\pm\beta_i), \pm2(\beta_j\pm\beta_i) |\, 2\leq i<j\leq 5\}
\subseteq\mathcal{\overline{W}}\setminus\mathcal{W}\subseteq\mathcal{R}(\mathfrak{h})$.
Considering the orthogonal decomposition \eqref{hspl} of $\mathfrak{h}$, it follows by Lemma~\ref{mixedroots} (ii) that
there is a $k\in\{0,1,2\}$ such that 
$\{\pm2(\alpha_1\pm\beta_i), \pm2(\beta_j\pm\beta_i) |\, 2\leq i<j\leq 5\}\subseteq\mathcal{R}(\mathfrak{h}_k)$.
Since $\alpha_1\in\hh_0\oplus\hh_1$ and  $\beta_i\in\hh_0\oplus\hh_2$, for $1\leq i\leq 5$, the only possible value is 
$k=0$. This implies that $\mathfrak{so}(10)\subseteq\mathfrak{h}_0$ and thus $p\geq 5$. We show that this is not possible.

Assuming that $p\geq 2$ and computing the scalar products between 
$\alpha_1+\beta_1+\cdots+\beta_5$ and $\alpha_2+\beta_1\pm(\beta_2+\beta_3)\pm(\beta_4+\beta_5)$,
we get the following values $\{a+\frac{7}{4},a+\frac{3}{4},a-\frac{1}{4}\}$, where $a:=\<\alpha_1,\alpha_2>$. By \eqref{prodscal}, 
we know that $\{a+\frac{7}{4},a+\frac{3}{4},a-\frac{1}{4}\}\subseteq\{0,\pm1\}$.
Hence, $a=-\frac{3}{4}$, which implies that $\norm(\alpha_1+\alpha_2)=-1$. 
Thus, the case $\norm(\alpha_1)=\frac{1}{4}$ may not occur.

We then have $\norm(\alpha_1)=\frac{3}{4}$ and $\norm(\beta_i)=\frac{1}{4}$, for $1\leq i\leq 5$. 
Again by computing all possible scalar products, we produce by \eqref{newroots}
the new roots $\{\pm2(\beta_i\pm\beta_j) |\, 1\leq i<j\leq 5\}\subseteq\mathcal{R}(\mathfrak{h})$.
By Lemma~\ref{mixedroots} (i) and (ii), there exists a $k\in\{0,1,2\}$ such that
$\{\pm2(\beta_i\pm\beta_j) |\, 1\leq i<j\leq 5\}$ are all roots of one of the
components $\mathfrak{h}_k$ of the orthogonal splitting $\mathfrak{h}=\mathfrak{h}_0\oplus\mathfrak{h}_1\oplus\mathfrak{h}_2$ 
given by \eqref{hspl}. As $\beta_i\in\hh_0\oplus\hh_2$ for $1\leq i\leq 5$,
it follows that $k\in\{0,2\}$. 
Thus, $\mathcal{R}(\mathfrak{h}_k)$ contains the whole system of roots of $\mathfrak{so}(10)$.
On the other hand, $\mathfrak{h}_0\oplus\hh_2\subseteq\mathfrak{so}(10)$.
Hence, there are two possibilities:
either $\mathfrak{h}_0=\mathfrak{so}(10)$ and $\mathfrak{h}_2=0$ or
$\mathfrak{h}_0=0$ and $\mathfrak{h}_2=\mathfrak{so}(10)$.

Let us first note that if $p\geq 2$, then $\<\alpha_i,\alpha_j>=-1$, for all $1\leq i<j\leq p$.
By computing the scalar products between the different roots containing 
$\alpha_i$, respectively $\alpha_j$, we obtain the following values:
$a:=\<\alpha_i,\alpha_j>\in\{a+\frac{5}{4},a+\frac{1}{4},a-\frac{3}{4}\}$, which
by \eqref{prodscal} must be contained into $\{0,\pm1\}$.
Hence, the only possible value is $a=-\frac{1}{4}$.

In the first case, $\mathfrak{h}_0=\mathfrak{so}(10)$ implies that $p\geq 5$, 
which by the above remark leads to the following contradiction: $\norm(\alpha_1+\cdots+\alpha_5)=-\frac{5}{4}<0$.

Thus, the second case $\mathfrak{h}_2=\mathfrak{so}(10)$ and $\mathfrak{h}_0=0$ must hold. 
We show that $p=1$. Assuming $p\geq 2$, we compute 
$\<\alpha_1+\beta_1+\cdots+\beta_5,\alpha_2+\beta_1-\beta_2-\cdots-\beta_5>=-1$, which by \eqref{newroots}
yields the new mixed root
$\alpha_1+\alpha_2+2\beta_1\in\mathcal{R}(\mathfrak{h})$, 
contradicting Lemma~\ref{mixedroots} (i) (since $\alpha_1+\alpha_2\in\mathfrak{h}_1$
and $\beta_1\in\hh_2$). Thus, $p=1$ and
$\mathfrak{h}_1=\mathfrak{u}(1)$.

Concluding, it follows that $\mathfrak{h}=\mathfrak{so}(10)\oplus\mathfrak{u}(1)$.
Therefore, $\mathcal{R}(\mathfrak{g})=\mathcal{W}\cup\mathcal{R}(\mathfrak{so}(10)\oplus\mathfrak{u}(1))$, 
is exactly the system of roots of $\mathfrak{e}_{6}$ (cf. \cite[p. 57]{a}),
hence $\mathfrak{g}=\mathfrak{e}_{6}$.
From \cite[Thm. 6.1]{a}, the Lie subgroup of $\mathrm{E}_6$ corresponding to
the above embedding of $\mathfrak{so}(10)\oplus\mathfrak{u}(1)$ in $\mathfrak{e}_{6}$
is $\mathrm{Spin}(10)\times \mathrm{U}(1)/\ZM_4$, showing that the only homogeneous manifold carrying a homogeneous
Clifford structure of rank $r=10$ is
the exceptional symmetric space
$(\mathbb{C}\otimes\mathbb{O})\mathbb{P}^2=
\mathrm{E}_6/(\mathrm{Spin}(10)\times\mathrm{U}(1)/\ZM_4)$.

(III) By Theorem~\ref{upperbd} (III), the maximal rank in this case
is $r=12$. For $r=12=2q$, from Proposition~\ref{wisotrop} (III),
there exist $\A:=\{\alpha_1,\ldots,\alpha_p\}$ and $\G:=\{\gamma_1,\ldots,\gamma_{p'}\}$
in $\tt^*$ with $\A=-\A$ and $\G=-\G$ such that
the set of weights of the isotropy representation is given by:
 $$\mathcal{W}(\mathfrak{m})=\A+\left\{\overset{6}{\underset{j=1}{\sum}}
\varepsilon_j\beta_j |\, \overset{6}{\underset{j=1}{\prod}}
\varepsilon_j=1\right\}_{\varepsilon\in\mathcal{E}_6}
\bigcup\G+\left\{\overset{6}{\underset{j=1}{\sum}}
\varepsilon_j\beta_j |\, \overset{6}{\underset{j=1}
{\prod}}\varepsilon_j=-1\right\}_{\varepsilon\in\mathcal{E}_6}$$
with $\#\mathcal{W}=(p+p')\cdot 2^{5}$, where one of $p$ or $p'$ might vanish, but
the vectors $\alpha_i$ and $\gamma_i$ are all non-zero.

Assume $p\neq 0$ (otherwise the same argument applies for $p'\neq 0$ by changing the sign of $\beta_1$) and denote by
\[\mathcal{W}:=\A+\left\{\overset{6}{\underset{j=1}{\sum}}
\varepsilon_j\beta_j |\, \overset{6}{\underset{j=1}{\prod}}
\varepsilon_j=1\right\}_{\varepsilon\in\mathcal{E}_6},\]
which is an admissible subsystem of roots as in Proposition~\ref{linalg} (III), with
$\#\mathcal{W}=p\cdot 2^{5}$.

We claim that the following inclusions hold:
$\mathcal{R}(\so(12)\oplus\su(2))\subseteq\overline{\mathcal{W}}\setminus\mathcal{W}$,
$\mathcal{R}(\mathfrak{e}_7)\subseteq\overline{\mathcal{W}}$.

Let us consider all the subsystems of roots of the following form 
$\{\sum_{j=1}^{4}\varepsilon_j \beta'_j\}\subset\mathcal{W}$, where $\beta'_1=\alpha_1$,
$\beta'_2=\beta_1\pm\beta_2$, $\beta'_3=\beta_3\pm\beta_4$,
$\beta'_4=\beta_5\pm\beta_6$ and such that
the number of minus signs in $\beta'_2, \beta'_3$ and $\beta'_4$ is even,
as well as all the other subsystems obtained by permuting
the subscripts $\{1,\dots,6\}$ of the vectors $\beta_j$.

By Lemma~\ref{typeI} (i), the Gram
matrix  $(\<\beta'_i, \beta'_j>)_{i,j}$ is either
$\frac{1}{4}\id_4$ or the matrix $M_0$ defined by \eqref{m4}. We further show that the second case can
not occur. Indeed, if this were the case, then at least two of the norms $\norm(\beta'_2)$, 
$\norm(\beta'_3)$ and $\norm(\beta'_4)$ are equal and after reordering, rescaling and sign change of the vectors $\beta_j$
we may assume $\beta'_2=\beta_1+\beta_2$, $\beta'_3=\beta_3+\beta_4$, $\beta'_4=\beta_5+\beta_6$ and 
$\norm(\beta'_3)=\norm(\beta'_4)=\frac{1}{8}$, $\<\beta'_3, \beta'_4>=\frac{1}{16}$. 
If we denote by $\beta''_j:=\beta'_j$, $j=1,2$ and 
$\beta''_3=\beta_3-\beta_4$, $\beta''_4=\beta_5-\beta_6$, 
then $\{\sum_{j=1}^{4}\varepsilon_j \beta''_j\}\subset\mathcal{W}$ 
is also a subsystem of roots. Since $\{\norm(\beta''_1), \norm(\beta''_2)\}=\{\frac{1}{4}, \frac{1}{8}\}$, 
it follows again by Lemma~\ref{typeI} (i) 
that the Gram matrix  $(\<\beta''_i, \beta''_j>)_{i,j}$ is equal to $M_0$ defined by \eqref{m4}, 
so that $\norm(\beta''_4)=\norm(\beta''_3)=\frac{1}{8}=\norm(\beta'_3)=\norm(\beta'_4)$.
Thus, $\norm(\beta_1+\beta_2)=\frac14$,
$\norm(\beta_{2j-1}+\beta_{2j})=\frac18$ for $j=2,3,4$, and
\beq \label{prs1}\<\beta_3+\beta_4,\beta_5+\beta_6>=\frac1{16}.\eeq
For every $3\le j<k\le 8$, there exists $l\in\{2,3,4\}$
such that $\{2l-1,2l\}\cap\{j,k\}=\emptyset$. Let $\{s,t\}$ denote the
complement of $\{1,2,j,k,2l-1,2l\}$ in $\{1,\ldots,8\}$.
The Gram matrix
of $\{\beta_1+\beta_2,
\beta_{2l-1}+\beta_{2l},\beta_j-\beta_k,\beta_s-\beta_t\}$ has at least two
different values on the diagonal. By
Lemma~\ref{typeI} (i) again, the remaining diagonal terms
$\norm(\beta_j-\beta_k)$ and $\norm(\beta_s-\beta_t)$ are both equal to
$\frac18$. Thus
$\norm(\beta_j-\beta_k)=\frac18$ for $3\le j<k\le 8$, and similarly
$\norm(\beta_j+\beta_k)=\frac18$ for $3\le j<k\le 8$. Thus
$\<\beta_j,\beta_k>=0$ for all $3\le j<k\le 8$, contradicting
\eqref{prs1}.

It then follows that the Gram matrix of any subsystem of roots
$\{\sum_{j=1}^{4}\varepsilon_j \beta'_j\}$ as above
is $(\<\beta'_i, \beta'_j>)_{i,j}=\frac{1}{4}\id_4$.
In particular, this shows that for any $1\leq i<j\leq 6$,
$\norm(\beta_i+\beta_j)=\norm(\beta_i-\beta_j)=\frac{1}{4}$, and
$\<\alpha_1,\beta_i+\beta_j>=\<\alpha,\beta_i-\beta_j>=0$,
implying that $\<\beta_i,\beta_j>=0$ and $\<\beta_i,\alpha_1>=0$.
For any $k\neq i,j$ we also have $\norm(\beta_i+\beta_k)=\norm(\beta_j+\beta_k)=\frac{1}{4}$,
which then yields $\norm(\beta_i)=\frac{1}{8}$ for $1\le i\le 6$.
Denoting by $\beta':=\sum_{i=1}^{4}\beta'_i$, we compute
$\<\beta'-2\beta'_i,\beta'>=\frac{1}{2}$, for $1\leq i\leq 4$.
By \eqref{newroots} we obtain the new roots $\pm2\beta'_i\in\mathcal{\overline{W}}\setminus\mathcal{W}$.
Since the argument is true for any such subsystem of roots,
we have the new roots
$\{\pm2\alpha_1,\pm2(\beta_i\pm\beta_j) |\, 1\leq i<j\leq 6\}\subseteq \mathcal{\overline{W}}\setminus\mathcal{W}$,
which, by the above orthogonality relations, build the system of roots of $\so(12)\oplus\su(2)$.
It thus follows  that
$\so(12)\oplus\su(2)\subseteq\hh$.
Furthermore, $\mathcal{R}(\so(12)\oplus\su(2))\cup\mathcal{W}=\mathcal{R}(\mathfrak{e}_7)$ 
(cf. \cite[p. 56]{a}), showing that $\mathfrak{e}_7\subseteq\mathfrak{g}$.

We claim that $p=1$. Assuming that $p\geq 2$, we consider $\alpha_2\in\A\setminus\{\pm\alpha_1\}$.
The previous argument also shows that $\<\beta_i,\alpha_2>=0$ for all $1\le i\le 6$.
Denoting by $a:=\<\alpha_1,\alpha_2>$, the scalar products between $\beta:=\alpha_1+\sum_{j=1}^{6}\beta_j$ and
$\alpha_2\pm(\beta_1+\beta_2)\pm(\beta_3+\beta_4)\pm(\beta_5+\beta_6)$ take four possible values: 
$\{a\pm\frac{3}{4},a\pm\frac{1}{4}\}$, thus contradicting \eqref{prodscal}, 
which only allows 3 different values for these scalar products. Hence, $p=1$, and similarly $p'\le 1$.

We further prove that $p'=0$. Assuming the contrary, there exists $\gamma_1\neq 0$ in $\G$.
The same arguments as before show that $\<\beta_i,\gamma_1>=0$ for $1\le i\le 6$ and $\norm(\gamma_1)=\frac14$.
Denoting by $a:=\<\alpha_1,\alpha_2>$, the set of scalar products between the unit vectors $\alpha_1+\sum_{j=1}^{6}\beta_j$ and
$\gamma_1+\beta_1-\beta_2\pm(\beta_3+\beta_4)\pm(\beta_5+\beta_6)$ equals 
$\{a,a\pm\frac12\}$. By \eqref{prodscal}, one has necessarily $a=0$, {\em i.e.} $\alpha_1\perp\gamma_1$.
We then denote by $\beta_7:=\frac12(\alpha_1+\beta_1)$ and $\beta_8:=\frac12(\alpha_1-\beta_1)$.
The above relations show that $\<\beta_i,\beta_j>=0$, $\norm(\beta_i)=\frac{1}{8}$
for $1\leq i<j\leq 8$ and 
$$\mathcal{W}(\mathfrak{m})=\left\{\overset{8}{\underset{j=1}{\sum}}
\varepsilon_j\beta_j |\, \overset{8}{\underset{j=1}{\prod}}
\varepsilon_j=1\right\}_{\varepsilon\in\mathcal{E}_8}.$$
Let $\beta:=\sum_{i=1}^{8}\beta_i$. Since $\<\beta,\beta-2\beta_i-2\beta_j>=\frac12$ for $i\neq j$ and 
$\<\beta-2\beta_i-2\beta_k,\beta-2\beta_j-2\beta_k>=\frac12$ for $i\neq j\neq k\neq i$,
by \eqref{newroots} we obtain that $\{\pm2(\beta_i\pm \beta_j) |\, 1\leq i<j\leq 8\}\in 
\overline{\mathcal{W}}\setminus\mathcal{W}\subset\mathcal{R}(\hh)$. This is exactly the 
system of roots of $\so(16)$. Thus, $\so(16)\subseteq\mathfrak{h}$.
Since $\so(16)$ is simple, the restriction to $\so(16)$ of the Clifford morphism $\rho_*:\hh\to\so(12)$ must vanish. 
Moreover, the restriction to $\so(16)$ of the representations $\lambda_\pm$ from Lemma \ref{isorep} vanish, too.
Indeed, $p=p'=1$ and $\K=\H$ so their complex dimensions equal 2. Thus the isotropy representation of $G/H$ would vanish
on $\so(16)\subset\hh$, a contradiction.

As $p'=0$, Lemma \ref{isorep} shows that the isotropy representation can be written $\mm=\mu_+\otimes_\H\lambda_+$, 
so as in \eqref{hspl} we can write 
\beq\label{hspl1}\hh=\hh_0\oplus\hh_1\oplus\hh_2\eeq
with $\hh_1:=\ker(\rho_*)$, $\hh_2:=\ker(\lambda_+)$ and $\hh_0=(\hh_1\oplus\hh_2)^\perp$.
Since $p=1$, it follows that $\hh_0\oplus\hh_1\subseteq\su(2)$.
On the other hand, $\hh_0\oplus\hh_2\subseteq\mathfrak{so}(12)$ and we have proved that
$\so(12)\oplus\su(2)\subseteq\hh$. Hence, we obtain 
$\hh_2=\so(12)$, $\hh_0=0$ and $\hh_1=\su(2)$.
In particular, we have $\mathfrak{h}=\so(12)\oplus\su(2)$, and $\mathcal{R}(\gg)=\mathcal{W}(\mm)\oplus \mathcal{R}(\hh)$
is isometric to the root system of $\mathfrak{e}_7$ (cf. \cite[p. 56]{a}). We conclude that
$M=\mathrm{E}_7/\mathrm{Spin}(12)\cdot\mathrm{SU}(2)$ by \cite[Thm 6.1]{a}.

(IV) For $r=16=2q$, it follows from Proposition~\ref{wisotrop} (IV) and Proposition~\ref{linalg} (III)-(IV) 
(for the extremal case $q=8$) that (up to a sign change for one of the vectors $\beta_i$) $p=1$, $p'=0$ and 
the set of weights of the isotropy representation is given by:
\[\mathcal{W}(\mathfrak{m})=\left\{\overset{8}{\underset{j=1}{\sum}}
\varepsilon_j\beta_j |\, \overset{8}{\underset{j=1}{\prod}}
\varepsilon_j=1\right\}_{\varepsilon\in\mathcal{E}_8}.\]

By Lemma~\ref{typeI} (b) it follows that the Gram matrix
$(\<\beta_i,\beta_j>)_{ij}$ is equal to $\frac{1}{8}\id_8$.
Thus, all roots in $\mathcal{W}$ have norm $1$ and $\<\beta,
\beta_i>=\frac{1}{8}$, where $\beta:=\sum_{i=1}^{8}\beta_i$.
This yields the following values for the scalar
products for all $i,j$ and $k$ mutually distinct : $\<\beta,
\beta-2\beta_i-2\beta_j>=\<\beta-2\beta_i-2\beta_j,
\beta-2\beta_i-2\beta_k>=\frac{1}{2}$ and all other scalar products of
roots in $\mathcal{W}$ are $0$. The new roots we 
obtain by \eqref{newroots} are then $\{\pm2(\beta_i\pm \beta_j) |\, 1\leq i<j\leq 8\}$,
which build the system of roots of $\so(16)$. Thus,
$\so(16)\subseteq\mathfrak{h}$.
As $p=1$ and $p'=0$ and $\so(16)\subseteq\mathfrak{h}$, it follows with the notation from
\eqref{hspl1} that $\so(16)\subseteq\mathfrak{h}_2$.
On the other hand, $\rho_*$ maps $\hh_0\oplus\hh_2$ one-to-one into $\so(16)$ and thus
we must have equality: $\hh_0\oplus\hh_2=\so(16)$.
Consequently, $\mathcal{R}(\mathfrak{g})=\mathcal{R}(\mathfrak{so}_{16})\cup \mathcal{W}$
is isometric to the system of roots of $\mathfrak{e}_8$ (cf. \cite[p. 56]{a}), showing that 
$\mathfrak{g}=\mathfrak{e}_8$ and $M=\mathrm{E}_8/(\mathrm{Spin}(16)/\mathbb{Z}_2)$ by \cite[Thm 6.1]{a}.
\end{proof}

As already mentioned in Section 1, similar methods could be used to examine the remaining
cases $r\leq 8$. However, the arguments tend to be much more intricate as fast as the rank 
decreases. In order to keep this paper at a reasonable length, we have thus decided to 
limit our study to the extremal cases of Theorem \ref{upperbd}.

\end{document}